\documentclass[11pt]{amsart}
\usepackage{graphicx}
\usepackage{latexsym}
\usepackage{amsfonts,amsmath,amssymb}

\def\pr{\textup{ P\/}}
\def\ex{\textup{E\/}}

\def\eps{\varepsilon}
\def\la{\lambda}
\def\a{\alpha}
\def\be{\beta}

\def\part{\partial}

\newcommand{\beq}{\begin{equation}}
\newcommand{\eeq}{\end{equation}}

\newtheorem{Theorem}{Theorem}[section]
\newtheorem{Lemma}[Theorem]{Lemma}

\newtheorem{Corollary}[Theorem]{Corollary}

\theoremstyle{remark}

\numberwithin{equation}{section}
\date{\today}
\begin{document}

\title[Exchange stable matchings]{On random exchange-stable matchings}

\author{Boris Pittel}
\address{Department of Mathematics, The Ohio State University, Columbus, Ohio 43210, USA}
\email{bgp@math.ohio-state.edu}



\begin{abstract} Consider the group of $n$ men and $n$ women, each with their own preference
list for a potential marriage partner. The stable marriage is a bipartite matching such that no unmatched pair 
(man, woman) prefer each other to their partners in the matching. Its non-bipartite version, with an even
number $n$ of members, is  known as the stable roommates problem.
Jose Alcalde introduced an alternative notion of exchange-stable, one-sided, matching:  no two 
 members prefer each other's partners to their own partners in the matching. 
Katarina Cechl\'arov\'a and David Manlove showed that the e-stable matching decision problem is $NP$-complete for both types of matchings. We prove that the expected number of e-stable matchings is 
asymptotic to $\left(\frac{\pi n}{2}\right)^{1/2}$ for two-sided case, and to $e^{1/2}$ for one-sided case. However, the standard deviation of this number exceeds $1.13^n$,
($1.06^n$ resp.). As an obvious  byproduct, there exist instances of preference lists with at least $1.13^n$ ($1.06^n$ resp.) e-stable matchings. The probability that there is no matching
which is  stable and e-stable is at least $1-e^{-n^{1/6+o(1)}}$, ($1-O(2^{-n/2})$ resp.).
\end{abstract}

\maketitle

\section{Introduction and main results} Consider the group of $n$ men and $n$ women, each member
each with their own preference list for a potential marriage partner. The stable marriage is a bipartite matching such that no unmatched pair (man,woman) prefer each other to their partners in the matching. A classic theorem, due to David Gale and Lloyd Shapley \cite{GalSha}, asserts that, given any system of preferences  there exists at least one stable marriage $M$. 
The proof of this fundamental theorem was based on analysis of a proposal algorithm: at each step, the
men
not currently on hold each make
a proposal to their best choice among women who haven't rejected them before, and the chosen woman either provisionally puts the man on hold or rejects him, based on comparison of him to her  current suitor if she has one already. The process terminates once every woman has a suitor, and the resulting bijection turns out to be stable. Of course, the roles can be reversed. In general, the two resulting matchings, $M_1$ and $M_2$ are different, one men-optimal/women-pessimal, another women-optimal/men-pessimal. The interested reader is encouraged to consult Dan Gusfield and Rob Irving \cite{GusIrv} for
a masterful, detailed analysis of the algebraic (lattice) structure of stable matchings set, and a collection of proposal algorithms for determination of the stable matchings in between the two extremal matchings $M_1$ and $M_2$.

A decade after the Gale-Shapley paper, McVitie and Wilson \cite{McVWil} developed an alternative, sequential, algorithm in which proposals by one side to another are made one at a time. This procedure delivers the same
matching as the Gale-Shapley algorithm; the overall number of proposals made, say by men to
women, is clearly the total rank of the women in the terminal matching. 

This purely combinatorial, numbers-free, description  
calls for a probabilistic analysis of the problem chosen uniformly at random among all the instances of
preference lists, whose total number is $(n!)^{2n}$.  In a pioneering paper \cite{Wil} Wilson reduced
the work of the sequential algorithm to a classic urn scheme (coupon-collector problem) and proved that the expected running time, whence the expected total rank of wives in the man-optimal
matching, is at most $nH_n\sim n\log n$, $H_n=\sum_{j=1}^n 1/j$.

Few years later  Don Knuth \cite{Knu}, among other results, found that, in fact, the expected running time
is asymptotic to $n\log n$, and also that the worst-case running time is $O(n^2)$, attributing the latter to an
unpublished work by J. Bulnes and J. Valdes. He also posed a series of open problems, one of them
on the {\it expected\/} number of the stable matchings. Don pointed out that an answer might be found via his formula for the probability $P(n)$ that a generic matching $M$ is stable:
\begin{equation}\label{Pn=}
P(n)=\overbrace {\idotsint}^{2n}_{\bold x,\,\bold y\in [0,1]^n}\,\prod_{1\le i\neq j\le n}
(1-x_iy_j)\, d\bold x d\bold y.
\end{equation}
And then the expected value of $S_n$, the total number of stable matchings, would then be determined from $\ex[S_n]=n! P(n)$. 

Following Don Knuth's suggestion, in \cite{Pit1} we used the equation \eqref{Pn=} to obtain an asymptotic formula $P(n)\sim\frac{e^{-1}n\log n}{n!}$,
which implied that $\ex[S_n]\sim e^{-1}n\log n$. We also found the integral formulas  for $P_k(n)$ ($P_{\ell}(n)$ resp.) the 
probability that the generic matching $M$ is stable {\it and\/} that the total man-rank $R(M)$
(the total woman-rank $Q(M)$ is $\ell$ resp.).
These integral formulas implied that with high probability (w.h.p. from now) 
for each stable matching $M$  the ranks $R( M)$, $Q(M)$ are 
between $(1+o(1))n\log n$
and  $(1+o(1))n^2/\log n$.  It followed, with some work, that w.h.p.
$R(M_1)\sim n^2/\log n$, $Q(M_1)\sim n\log n$ and 
$R(M_2)\sim n\log n$, $Q(M_2)\sim n^2/\log n$. 
In particular, w.h.p. $R(M_j)Q(M_j)$ $\sim n^3$, ($j=1,2$).  

Spurred by these results, in \cite{Pit2} we studied the likely behavior of the full random set $\{(R( M),Q(M))\}$, where $M$ runs through all stable matchings for the random instance of preferences.  We proved a {\it law of hyperbola\/}:
for every $\la\in (0,1/4)$,  quite surely (q.s) $\max_{M}|n^{-3}Q(M)R( M)-1|\le n^{-\la}$;  ``quite surely'' means with probability $1-O(n^{-K})$, for every $K$, a notion
introduced by Knuth, Motwani and the author \cite{KnuMotPit}. 
Furthermore, q.s. $S_n\ge n^{1/2-o(1)}$, a significant  improvement of  the logarithmic bound in \cite{KnuMotPit}, but still far below $n\log n$, the asymptotic order of $\ex[S_n]$.

Thus, for a large number of participants, a typical instance of the preference lists has multiple stable
matchings very nearly obeying the preservation law for the product of the total man-rank and
the total woman-rank. 

Eight years ago  with Craig Lennon \cite{LenPit} we extended the techniques in \cite{Pit1}, \cite{Pit2} to
show that $\ex[S_n^2]\sim (e^{-2}+0.5e^{-3})(n\log n)^2$. Combined with $\ex[S_n]\sim e^{-1}n\log n$, this
result implied that $S_n$ is of order $n\log n$ with probability $0.84$, at least. 

A recent breakthrough study of the stable matchings in unbalanced settings by Itai Ashlagi, Yash Kanoria and Jacob Leshno \cite{AshKanLes} (see our follow-up analysis in \cite{Pit4}) proves that the probabilistic aspects of this
classic combinatorial scheme continue to be a goldmine of interesting problems. In fact, the recent monograph by David Manlove \cite{Man} covers an astonishing variety of new matching models and
algorithms, making some of them ripe for probabilistic study as well. In particular,
David discussed an alternative notion of stability suggested by Jose Alcalde \cite{Alc}: a matching $M$  is called 
exchange-stable (e-stable), if no two members prefer each other's partners to their own partners under $M$.
Actually, Jose dealt with the one-sided matchings, so called roommates assignment problem, but
the notion of e-stability makes sense for two-sided matchings as well. 

Somehow, this elegant scheme reminded the author of the stochastic model \cite{Pit5} (see also our appendix to Michael L. Tsetlin's book \cite{Tse}). In that model
 a randomly chosen pair of city dwellers, currently housed in the residential areas $j_1$ and $j_2$,
and employed by the plants $i_1$ and $i_2$, exchange their residencies with probability $\pi(t_{i_1,j_1}) 
\pi(t_{i_2,j_2})$, $t_{i,j}$ being the commute time from $j$ to $i$, and $\pi(t)$ monotone increasing with $t$. For the
large total population $n$, the limiting matrix of the numbers $x_{i,j}$ of persons working in the $i$-th plant and living in the $jt$-th residential district maximizes the weighted entropy $\sum_{i,j} x_{i,j}\log
\frac{\nu_{i,j}}{x_{i,j}}$ subject to the row and column constraints. The numbers $\nu_{i,j}=(a_i/\pi(t_{i,j})) /(\sum_k 1/\pi(t_{i,k}))$, $a_i$ being the total roster of the plant $i$, can be interpreted as an ideal allocation of $n$ members among the residential areas, when they do not have to compete for the limited
capacities of the residential areas.

Katarina Cechl\'arov\'a and David Manlove \cite{CecMan} showed that, in sharp contrast to the classic stable matchings, the e-stable matching decision problem is $NP$-complete for both types of matchings. 
(It is a good place to mention that the ``fundamental proposal algorithm'' constructed
by Rob Irving for the one-sided stable matchings has $O(n^2)$ worst-case running time,
\cite{GusIrv}.) This surprising result in \cite{CecMan} prodded us to look at the {\it likely\/} behavior of the e-stable matchings.

We prove that the expected number of e-stable matchings is asymptotic to $\left(\frac{\pi n}{2}\right)^{1/2}$,
definitely smaller than $e^{-1}n\log n$ for the classic stable
matchings \cite{Pit0}, but in the same league qualitatively. Somehow we felt that the second moment
of the number of e-stable matchings would grow like $n^{\gamma}$, for some $\gamma\ge 1$ of course.
That has been the case so far with the stable matchings, bipartite and non-bipartite,  and also the
stable partitions introduced and studied, algorithmically, by Jimmy Tan \cite{Tan}, \cite{Tan1}; see \cite{Pit3} for the probabilistic results.  

However, as a possible reflection of substantial algorithmic complexity of Alcalde's model, this second order moment exceeds $1.28^n$, i.e. grows exponentially fast. Consequently the
standard deviation of the number of e-stable matchings exceeds $1.13^n$, signaling that the discernible
right tail of the distribution of that number is much longer than the left tail.
As an obvious byproduct, we claim existence  of preference lists with at least $1.13^n$ e-stable matchings.
Similar bounds for the stable marriages have long been known, see \cite{Man}, Section 2.2.2, for discussion and references. However those bounds were obtained via explicit constructions of the preference
lists having exponentially many stable matchings. By the very nature of
the probabilistic method we use, our claim is purely existential.

We also consider the one-sided e-stable matchings on the set of $n$ (even) members, under the assumption that the instance of $n$ preference lists, each with $n-1$ positions, is chosen uniformly
at random among all $[(n-1)!]^n$ such instances.  We had proved
that the expected value of the number of the one-sided stable matchings converges to the finite $e^{1/2}$ , and that the standard deviation of this number is $\sim\left(\frac{\pi n}{4e}\right)^{1/4}$, approaching infinity
moderately fast. In this paper we show that the expected value of the number of the e-stable, 
one-sided matchings, is exactly the same, so is $e^{1/2}$ in the limit. However, its standard deviation
is exponentially large, $1.06^n$ at least, in qualitative harmony with the two-sided e-stable matchings. Can
the overwhelming ``asymmetry'' of the distribution of the number of e-stable matchings be a hint that, with probability approaching $1$, at least one such matching exists? ``Overwhelming'' is a key word here: 
for the classic stable matching problem, when the standard deviation is of order $n^{1/4}$ only, 
the limiting probability that  a solution exists is below $0.5e^{1/2}<1$, \cite{IrvPit}. 

We show that q.s. uniformly for every e-stable matching $M$, two-sided or one-sided, the arithmetic average of the partners' ranks is asymptotic to $n^{1/2}$, just like the stable one-sided matchings on $[n]$,
\cite{Pit2}.

Katarina Cechl\'arov\'a and David Manlove \cite{CecMan}, Rob Irving \cite{Irv1}, Eric McDermid, Christine Cheng and Ichiro Suzuki \cite{McDCheSuz} studied matchings that are doubly stable, i.e. both classically stable and
(coalition)-exchange-stable. We are back to e-stability when coalition size is $2$ only. It was proved in
\cite{CecMan} that, for unrestricted coalition size, a doubly stable marriage exists only if a stable marriage
is unique. Strikingly,
a two-sided instance does not necessarily admit a stable matching which is simply man-exchange
stable, \cite{Irv1}. We prove that this kind of incompatibility holds for almost all large-size instances of
two-sided and one-sided preference lists. More precisely, the probability that there is no doubly stable matching is at least $1-e^{-n^{1/6-o(1)}}$ (two-sided case), and $1-O(2^{-n/2})$ (one-sided case).

\section{Basic identities and bounds}
\subsection{Two-sided matchings}
 Consider an instance of the $n$-men/$n$-women matching problem under preferences chosen uniformly
 at random among all $(n!)^{2n}$  such instances. We need to derive the integral (Knuth-type) formulas for $\pr(M)$ the probability that a matching $M$ is exchange-stable (e-stable), $\mathcal P(M)$
the probability that $M$ is both e-stable and stable, 
and $\pr(M_1,M_2)$ the probability that two matchings $M_1\neq M_2$ are each e-stable.
It is convenient to view $M$ as a bijection from, say, the men set to the women set. 

Observe that the uniformly random instance of the $2n$ preference  lists can be generated as follows. 
Introduce two $n\times n$ arrays of the $2n^2$ independent random variables,  
$X_{i,j}$ ($1\le i,\, j\le n$) and $Y_{i,j}$ ($1\le i,\, j\le n$),  each distributed uniformly on $[0,1]$. Assume that each man $i$ (woman $j$ resp.) 
ranks the women (men resp.) in increasing order of the variables $X_{i,k}$, $1\le k\le n$,
($Y_{\ell,j}$, $1\le\ell\le n$).  Each of the resulting $2n$ orderings is uniform, and all the orderings are independent. 

\begin{Lemma}\label{P(Mest)=} Let $(a,b)$ stand for a generic, unordered, pair of distinct elements of $[n]$. Then, for every matching $M$,
\begin{equation}\label{P(Mest)=int^2}
\begin{aligned}
\pr(M)&=\left(\idotsint\limits_{\bold x\in [0,1]^{n}}\prod_{(a,b)}(1-x_ax_b)\,d\bold x\right)^2,\\
\mathcal P(M)&=\idotsint\limits_{\bold x,\,\bold y\in [0,1]^{n}}\mathcal P(M |\bold x,\bold y)\,d\bold xd\bold y,
\end{aligned}
\end{equation}
where
\begin{equation}\label{mathcalP(M|)=}
\begin{aligned}
\mathcal P(M |\bold x,\bold y)&=\prod_{(i_1,i_2)}\!\!\pr\Bigl(\bigl\{X_{i_1,M(i_2)}<x_{i_1},\,Y_{i_1,M(i_2)}<y_{M(i_2)}\bigr\}^c\\
&\qquad\qquad\cap\bigl\{X_{i_2,M(i_1)}<x_{i_2},\,Y_{i_2,M(i_1)}<y_{M(i_1)}\bigr\}^c\\
&\qquad\qquad\cap\bigl\{X_{i_1,M(i_2)}<x_{i_1},\,X_{i_2,M(i_1)}<x_{i_2}\bigr\}^c\\
&\qquad\qquad\cap\bigl\{Y_{i_1,M(i_2)}<y_{M(i_2)},\,Y_{i_2,M(i_1)}<y_{M(i_1)}\bigr\}^c\Bigr).
\end{aligned}
\end{equation}
\end{Lemma}

\begin{proof} $M$ is e-stable if and only if 
we have
\begin{align*}
&\forall\, (i_1, i_2),\,X_{i_1, M(i_1)}> X_{i_1, M(i_2)}\Longrightarrow X_{i_2, M(i_2)}<X_{i_2, M(i_1)},\\
&\forall\, (j_1, j_2),\, Y_{M^{-1}(j_1), j_1}>Y_{M^{-1}(j_2), j_1}\Longrightarrow Y_{M^{-1}(j_2), j_2}<
Y_{M^{-1}(j_1), j_2}.
\end{align*}
We can say that a pair of men $(i_1,\,i_2)$ (a pair of women $(j_1,j_2)$ resp.) blocks the matching $M$
, i.e. prevents $M$ from being e-stable,  if $X_{i_1, M(i_1)}> X_{i_1, M(i_2)}$,  
$X_{i_2, M(i_2)} >X _{i_2, M(i_1)}$ (if $Y_{M^{-1}(j_1), j_1} > Y_{M^{-1}(j_2), j_1}$, 
$Y_{M^{-1}(j_2), j_2}> Y_{M^{-1}(j_1), j_2}$ resp.). So $M$ is e-stable if no two men block $M$
and no two women block $M$.

By independence of the matrices $\{X_{i,j}\}$ and $\{Y_{i,j}\}$, the $\binom{n}{2}$ first-line events and the $\binom{n}{2}$ second-line events are collectively independent.  Furthermore, conditioned on
$\{X_{i, M(i)}=x_i,\,i\in [n]\}$ (on $\{Y_{M^{-1}(j), j}=y_j,\, j\in [n]\}$ resp.) the $\binom{n}{2}$ events in the
first (second resp.) line are independent among themselves. Therefore
\begin{equation}\label{P(Mest|x,y)=}
\begin{aligned}
&\pr\Bigl(M\text{ is e-stable}\,\boldsymbol |\, X_{i, M(i)}=x_i,\,Y_{M^{-1}(j), j}=y_j,\,i,j\in [n]\Bigr)\\
&\qquad\quad\qquad=\prod_{(i_1, i_2)}\!\!\pr\Bigl(\bigl\{X_{i_1, M(i_2)}< x_{i_1},\, X_{i_2, M(i_1)}<x_{i_2}\bigr\}^c\Bigr)\\
&\qquad\qquad\qquad\cdot\!\!\prod_{(j_1,j_2)}\!\!\pr\Bigl(\bigl\{Y_{M^{-1}(j_2), j_1}<y_{j_1},\,Y_{M^{-1}(j_1), j_2}<y_{j_2}\bigr\}^c\Bigr)\\
&\qquad\qquad\quad=\prod_{(i_1, i_2)}\bigl(1-x_{i_1} x_{i_2}\bigr)\cdot \prod_{(j_1, j_2)}\bigl(1-y_{j_1} y_{j_2}\bigr).
\end{aligned}
\end{equation}
Integrating both sides for $\bold x,\,\bold y \in [0,1]^n$ we obtain the top formula in \eqref{P(Mest)=int^2}.
The proof of the bottom formula is similar, as $\mathcal P(M |\bold x,\bold y)$ is the probability that $M$
is e-stable and stable, conditioned on $\{X_{i, M(i)}=x_i,\,Y_{M^{-1}(j), j}=y_j,\,i,j\in [n]\}$. Of course,
\begin{equation}\label{mathcal P(M|)<P(M|)}
\mathcal P(M|\bold x,\bold y)\le \prod_{(i_1, i_2)}\bigl(1-x_{i_1} x_{i_2}\bigr)\cdot \prod_{(j_1, j_2)}\bigl(1-y_{j_1} y_{j_2}\bigr).
\end{equation}
\end{proof}
Next, introduce $Q(M)$ and $R(M)$, the total sum of all wives' ranks and the total sum of all husbands'
ranks on the preference lists of their spouses under matching $M$. Using $\chi(A)$ to denote the indicator of an event $A$, we have
\begin{align*}
Q(M)&=n+\sum_{i, j\neq M(i)}\chi\bigl(X_{i, j}< X_{i, M(i)}\bigr)\\
&=n+\sum_{\{i_1,\,i_2\}}\chi\bigl(X_{i_1, M(i_2)}<X_{i_1, M(i_1)}\bigr)\\
&=n+\sum_{(i_1, i_2)}\Bigl[\chi\bigl(X_{i_1, M(i_2)}<X_{i_1, M(i_1)}\bigr)+
\chi\bigl(X_{i_2, M(i_1)}<X_{i_2, M(i_2)}\bigr)\Bigr],
\end{align*}
and likewise
\begin{align*}
R(M)=n+\sum_{(j_1, j_2)}\Bigl[\chi\bigl(Y_{M^{-1}(j_2), j_1}&<Y_{M^{-1}(j_1), j_1}\bigr)\\
&+\chi\bigl(Y_{M^{-1}(j_1), j_2}<Y_{M^{-1}(j_2), j_2}\bigr)\Bigr].
\end{align*}

\noindent For $n\le k,\ell \le n^2$, let $\pr_{k,\ell}(M):=\!\!\pr(M\text{ is e-stable}, Q(M)=k, R(M)=\ell)$. 
\begin{Lemma}\label{Pkell(M)} Using notation $\bar z=1-z$,
\begin{equation}\label{Pkell(M)=} 
\begin{aligned}
\pr_{k,\ell}(M)&=\idotsint\limits_{\bold x\in [0,1]^{n}}[\xi^{k-n}]\prod_{(a,b)}
\bigl(\bar x_a\bar x_b+\xi x_a\bar x_b+\xi \bar x_a x_b\bigr)\,d\bold x\\
&\times\idotsint\limits_{\bold y\in [0,1]^{n}} [\eta^{\ell-n}]\prod_{(c,d)}
\bigl(\bar y_c\bar y_d+\eta y_c\bar y_d+\eta \bar y_c y_d\bigr)\,d\bold y.
\end{aligned}
\end{equation}
Thus $\pr_{k,\ell}(M)$ does not depend on $M$.
\end{Lemma}
\begin{proof} First of all,
we have
\begin{align*}
&\qquad\qquad\qquad \pr_{k,\ell}(M)=[\xi^k\eta^{\ell}]\,\ex\Bigl[\xi^{Q(M)} \eta^{R(M)} \chi(M\text{ is e-stable})\Bigr],\\
&
\chi\bigl(M\text{ is e-stable}\bigr)=\!\prod_{(i_1 ,i_2)}\!\!\chi\Bigl(\bigl\{X_{i_1, M(i_2)}<X_{i_1, M(i_1)},\,X_{i_2, M(i_1)}<X_{i_2, M(i_2)}\bigr\}^c\Bigr)\\
&\quad\times \prod_{(j_1, j_2)} \chi\Bigl(\bigl\{Y_{M^{-1}(j_2), j_1}<Y_{M^{-1}(j_1), j_1},\,Y_{M^{-1}(j_1), j_2}
<Y_{M^{-1}(j_2), j_2}\bigr\}^c\Bigr).
\end{align*}
So, conditioning on $\{X_{i,\,M(i)}\}_{i\in [n]}=\bold x$ and $\{Y_{M^{-1}(j),j}\}_{j\in [n]}=\bold y$ respectively, we have
\begin{align*}
&\ex\Bigl[\xi^{Q(M)}\!\!\prod_{(i_1 ,i_2)}\!\!\chi\Bigl(\bigl\{X_{i_1, M(i_2)}<X_{i_1, M(i_1)},\,X_{i_2, M(i_1)}<X_{i_2, M(i_2)}\bigr\}^c\Bigr)
\Big|\,\, \bold x\Bigr]\\
&=\xi^n\prod_{(i_1,i_2)}\ex\Bigl[\xi^{\chi\bigl(X_{i_1, M(i_2)}<x_{i_1}\bigr)+\chi\bigl(X_{i_2, M(i_1)}
<x_{i_2}\bigr)}\\
&\qquad\qquad\qquad\qquad\qquad\cdot\,\chi\bigl(\{X_{i_1, M(i_2)}<x_{i_1},\,X_{i_2, M(i_1)}< x_{i_2}\}^c\bigr)\Bigr]\\
&\qquad\qquad\quad=\xi^n\prod_{(i_1, i_2)}\bigl(\bar x_{i_1}\bar x_{i_2} +\xi x_{i_1}\bar x_{i_2} +\xi \bar x_{i_1} x_{i_2}\bigr),
\end{align*}
as $\pr(X_{i_1, M(i_2)}<x_{i_1})=x_{i_1}$ and $\pr(X_{i_2, M(i_1)}< x_{i_2})=x_{i_2}$. Likewise
\begin{align*}
&\ex\Bigl[\eta^{R(M)}\!\!\prod_{(i_1 ,i_2)}\!\!\chi\Bigl(\bigl\{Y_{M^{-1}(j_2), j_1}<Y_{M^{-1}(j_1), j_1},\,Y_{M^{-1}(j_1), j_2}
<Y_{M^{-1}(j_2), j_2}\bigr\}^c\Bigr)\Big|\,\, \bold y\Bigr]\\
&\qquad\qquad\qquad\,\,\,=\eta^n\prod_{(j_1, j_2)}\bigl(\bar y_{j_1}\bar y_{j_2} +\eta y_{j_1}\bar y_{j_2} +\eta \bar y_{j_1} y_{j_2}\bigr).
\end{align*}
Integrating the two conditional expectations over $\bold x\in [0,1]^n$ and $\bold y\in [0,1]^n$ respectively,
and multiplying the integrals, we obtain
\begin{align*}
\ex\Bigl[\xi^{Q(M)} \eta^{R(M)}\chi(M\text{ is e-stable})\Bigr]&=\xi^n
\idotsint\limits_{\bold x\in [0,1]^{n}}\!\prod_{(a,b)}
\bigl(\bar x_a\bar x_b+\xi x_a\bar x_b+\xi \bar x_a x_b\bigr)\,d\bold x\\
&\times\eta^n\idotsint\limits_{\bold y\in [0,1]^{n}}\!\prod_{(c,d)}
\bigl(\bar y_{c}\bar y_{d} +\eta y_{c}\bar y_{d} +\eta \bar y_{c} y_{d}\bigr)\,d\bold y.
\end{align*}
This identity is equivalent to \eqref{Pkell(M)=}.
\end{proof}

Let $M_1\neq M_2$ be two generic matchings. Together $M_1$ and $M_2$ determine a bipartite graph $G(M_1,M_2)$ on the vertex set $[n]\times [n]$, 
 with the edge set $E$ formed by the man-woman pairs $(i,j)\in M_1\cup M_2$.
Each component of $G(M_1,M_2)$ is either an edge $e\in M_1\cap M_2$,  or a vertex-wise alternating circuit of even length at least $4$, in which the edges from $M_1$ and $M_2$ alternate as well. So the edge set for all these circuits is the symmetric difference $M_1\Delta M_2$. The vertex set $V(M_1\Delta M_2)$
is the union of the
men set $\mathcal N$ and the women set $\mathcal N'$, where $|\mathcal N|=|\mathcal N'|=: \nu$, 
and
\[
I:=\mathcal N^c=\{i: M_1(i)=M_2(i)\},\quad J:=(\mathcal N')^c=\{j: M_1^{-1}(j)=M_2^{-1}(j)\};
\]
$i\in I$ iff $j\in J$, where $j$ is the common value of $M_1(i)$, $M_2(i)$.
\begin{Lemma}\label{P(M1,M2est)} 
Denoting
$
\bold x_1=\{x_{i,1}:\,i\in [n]\}, \quad \bold x_2=\{x_{i, 2}:\,i\in [n]\},\quad x_{i, 1}= x_{i, 2}\,\,\text{ for }\,\,i\in I$,
and $\bold x_2^*= \{x_{i,2}:\,i\in \mathcal N\}$, we have 
\begin{equation*}
\begin{aligned}
P(M_1,M_2)&\ge\left(\,\,\,\idotsint\limits_{\bold x_1\in [0,1]^n,\,\,\bold x_2^*\in [0,1]^{\nu}}\!\!\!\!\!\!\!\!\!
f(\bold x_1,\bold x_2)\,d\bold x_1 d\bold x_2^*\right)^2,\\
f(\bold x_1,\bold x_2)&=\prod_{(i_1,i_2):\,i_1,i_2\in \mathcal N}\!\!\!\!\bigl(1-x_{i_1, 1}\,x_{i_2, 1}\bigr)\bigl(1-x_{i_1, 2}\,x_{i_2, 2}\bigr)\\
&\quad\times\!\!\prod_{i_1\in I,\, i_2\in \mathcal N}\bigl(1-x_{i_1, 1}\,x_{i_2, 1}\bigr)\bigl(1-x_{i_1, 2}\,x_{i_2, 2}\bigr)\\
&\quad\times\!\!\!\!\!\prod_{(i_1,i_2):\,i_1,i_2\in I}\!\!\!\bigl(1-x_{i_1,1}x_{i_2,1}\bigr).
\end{aligned}
\end{equation*}
\end{Lemma}
\begin{proof} We need to characterize e-stability of the matchings $M_1$ and $M_2$ in terms of the matrices $\{X_{i, j}\}$ and $\{Y_{i, j}\}$. Observe first that $X_{i, M_1(i)} = X_{i, M_2(i)}$ for $i\in I$ and $Y_{M_1^{-1}(j), j}=Y_{M_2^{-1}(j).j}$ for $j\in J$.

The matchings
$M_1$ and $M_2$ are both e-stable if and only if none of the pairs of men $(i_1, i_2)$ and none of the pairs of women $(j_1, j_2)$  blocks either $M_1$ or $M_2$. Introduce the events
\begin{align*}
A_{(i_1, i_2)}&=\!\bigl\{\!(i_1, i_2)\text{ blocks neither }M_1\text{ nor }M_2\bigr\},\\
B_{(j_1, j_2)}&=\!\bigl\{\!(j_1, j_2)\text{ blocks neither }M_1\text{ nor }M_2\bigr\}.
\end{align*}
The $\binom{n}{2}$ events $A_{(i_1,i_2)}$ are independent of the $\binom{n}{2}$ events $B_{(j_1,j_2)}$.
Furthermore, conditioned on the values $X_{i,M_1(i)}=x_{i,1}$, $X_{i,M_2(i)}=x_{i,2}$,  so that $x_{i, 1}=x_{i, 2}$ for $i\in I$
($Y_{M_1^{-1}(j),j}=y_{j,1}$, $Y_{M_2^{-1}(j),j}=y_{j,2}$  so that $y_{j, 1}=y_{j, 2}$ for $j\in J$ resp.), 
 the events $A_{(i_1,i_2)}$ ($B_{(j_1,j_2)}$ resp.) are independent among themselves. Introducing (in addition to $\bold x_1$, $\bold x_2$, $\bold x_2^*$)  the vectors $\bold y_1=\{y_{j, 1}:\,j\in [n]\}$, $\bold y_2=\{y_{j, 2}:\,j\in [n]\}$ and $\bold y_2^*=\{y_{j,2}\,:\,j\in \mathcal N'\}$, we have
 \begin{equation}\label{cond}
 \begin{aligned}
 &\qquad\qquad\pr\Bigl(M_1, M_2\text{ are e-stable }\boldsymbol |\,\bold x_1, \bold x_2, \bold y_1, \bold y_2\Bigr)\\
 &=
 \prod_{(i_1, i_2)}\!\!\pr\Bigl(A_{(i_1, i_2)}\boldsymbol |\,\bold x_1, \bold x_2\Bigr)
 \cdot \prod_{(j_1, j_2)}\!\!\pr\Bigl(B_{(j_1, j_2)}\boldsymbol |\,\bold y_1, \bold y_2\Bigr).
\end{aligned}
\end{equation}

\noindent Consider $\pr\Bigl(A_{(i_1, i_2)}\boldsymbol |\,\bold x_1, \bold x_2\Bigr)$. 

{\bf (1)\/} $M_1(i_1)\neq M_2(i_1)$, $M_1(i_2)\neq M_2(i_2)$. Then
\begin{align*}
\pr\Bigl(A_{(i_1, i_2)}\boldsymbol |\,\bold x_1, \bold x_2\Bigr)&=1-\pr\Bigl(X_{i_1, M_1(i_2)}< x_{i_1,1},\,
X_{i_2, M_1(i_1)} <x_{i_2, 1}\Bigr)\\
&\quad-\pr\Bigl(X_{i_1, M_2(i_2)} < x_{i_1,2},\,X_{i_2, M_2(i_1)}<x_{i_2, 2}\Bigr)\\
&\quad+\pr\Bigl(\{X_{i_1, M_1(i_2)}< x_{i_1,1},\,
X_{i_2, M_1(i_1)} <x_{i_2, 1}\}\\
&\qquad\quad\text{ and }\{X_{i_1, M_2(i_2)} < x_{i_1,2},\,X_{i_2, M_2(i_1)}<x_{i_2, 2}\}\Bigr)\\
&=1-x_{i_1, 1}\,x_{i_2, 1} - x_{i_1, 2}\, x_{i_2, 2} +x_{i_1,1} x_{i_1, 2} \cdot x_{i_2, 1} x_{i_2, 2}\\
&=\bigl(1-x_{i_1, 1}\,x_{i_2, 1}\bigr)\bigl(1-x_{i_1, 2}\,x_{i_2, 2}\bigr).
\end{align*}
{\bf (2)\/}  $M_1(i_1)=M_2(i_1)$, $M_1(i_2)\neq M_2(i_2)$. In this case $x_{i_1,1}=x_{i_1,2}$, and
\[
X_{i_2, M_1(i_1)}< x_{i_2,1},\,X_{i_2, M_2(i_1)}< x_{i_2,2}\Longleftrightarrow X_{i_2, M_1(i_1)} < x_{i_2,1}
\wedge x_{i_2, 2},
\]
and therefore
\begin{align*}
\pr\Bigl(A_{(i_1, i_2)}\boldsymbol |\,\bold x_1, \bold x_2\Bigr)
&=1-x_{i_1, 1}\,x_{i_2, 1} - x_{i_1, 2}\, x_{i_2, 2} +x_{i_1,1} x_{i_1, 2} \,(x_{i_2, 1}\wedge x_{i_2, 2})\\
&\ge \bigl(1-x_{i_1, 1}\,x_{i_2, 1}\bigr)\bigl(1-x_{i_1, 2}\,x_{i_2, 2}\bigr).
\end{align*}
 
{\bf (3)\/} Finally, if  $M_1(i_1)=M_2(i_1)$, $M_1(i_2)=M_2(i_2)$, then
\begin{align*}
\pr\Bigl(A_{(i_1, i_2)}\boldsymbol |\,\bold x_1, \bold x_2\Bigr)
&=1-x_{i_1, 1}\,x_{i_2, 1} - x_{i_1, 2}\, x_{i_2, 2} +(x_{i_1,1}\wedge x_{i_1, 2}) \,(x_{i_2, 1}\wedge x_{i_2, 2})\\
&=1-x_{i_1, 1}\,x_{i_2, 1}.
\end{align*}
Therefore
\begin{equation}\label{prodP(A)}
\begin{aligned}
 \prod_{(i_1, i_2)}\!\!\pr\Bigl(A_{(i_1, i_2)}\boldsymbol |\,\bold x_1, \bold x_2\Bigr)&\ge 
 \prod_{(i_1,i_2):\,i_1,i_2\in \mathcal N}\!\!\!\!\bigl(1-x_{i_1, 1}\,x_{i_2, 1}\bigr)\bigl(1-x_{i_1, 2}\,x_{i_2, 2}\bigr)\\
&\quad\times\!\!\prod_{i_1\in I,\, i_2\in \mathcal N}\bigl(1-x_{i_1, 1}\,x_{i_2, 1}\bigr)\bigl(1-x_{i_1, 2}\,x_{i_2, 2}\bigr)\\
&\quad\times\!\!\!\!\!\prod_{(i_1,i_2):\,i_1,i_2\in I}\!\!\!\bigl(1-x_{i_1,1}x_{i_2,1}\bigr).
\end{aligned}
\end{equation}
Similarly
\begin{equation}\label{prodP(B)}
\begin{aligned}
 \prod_{(j_1, j_2)}\!\!\pr\Bigl(B_{(j_1, j_2)}\boldsymbol |\,\bold y_1, \bold y_2\Bigr)&\ge 
 \prod_{(j_1,j_2):\,j_1,j_2\in \mathcal N'}\!\!\!\!\bigl(1-y_{j_1, 1}\,y_{j_2, 1}\bigr)\bigl(1-y_{j_1, 2}\,y_{j_2, 2}\bigr)\\
&\quad\times\!\!\prod_{j_1\in J,\, j_2\in \mathcal N'}\bigl(1-y_{j_1, 1}\,y_{j_2, 1}\bigr)\bigl(1-y_{j_1, 2}\,y_{j_2, 2}\bigr)\\
&\quad\times\!\!\!\!\!\prod_{(j_1,j_2):\,j_1,j_2\in J}\!\!\!\bigl(1-y_{j_1,1}y_{j_2,1}\bigr).
\end{aligned}
\end{equation}
Plugging the formulas \eqref{prodP(A)} and \eqref{prodP(B)} into \eqref{cond}, integrating
over $\bold x_1,$\linebreak $\bold x_2^*,\,\bold y_1,\,\bold y_2^*$, {\it and\/} using $|\mathcal N|=|\mathcal N'|$, $|I|=|J|$,  we complete the proof.
\end{proof}
The integral identities in Lemma \ref{P(M1,M2est)}, Lemma \ref{Pkell(M)} and Lemma \ref{P(M1,M2est)} turn out quite 
amenable to a sharp asymptotic analysis. 

\subsection{One-sided matchings}
 Consider an instance of the one-sided matching problem on the set $[n]$, $n$ even,  chosen uniformly
 at random from among all $((n-1)!^{n}$  such instances. Let $\pr(M)$ be the probability that a generic matching $M$ e-stable, and let $\pr(M_1,M_2)$ the probability that two generic matchings $M_1\neq M_2$ are both e-stable. In addition, introduce $\mathcal P(M)$ the probability that $M$ is doubly stable, i.e.
both exchange-stable and stable.

Like two-sided case,  the uniformly random instance of the $n$ preference  lists can be generated as follows. 
Introduce  the array of $n(n-1)$ independent, $[0,1]$-Uniforms $X_{i,j}$ ($1\le i\neq j\le n$) and $Y_{i,j}$ ($1\le i\neq j\le n$),  each distributed uniformly on $[0,1]$. Assume that each member $i$  
ranks other members in increasing order of the variables $X_{i,k}$, $k\neq i$.
 Each of the resulting $n$ preference lists is uniform, and all the lists are independent. 

\begin{Lemma}\label{P(Mest)='} For every matching $M$,
\begin{equation}\label{P(Mest)=int^2'}
\begin{aligned}
\pr(M)&=\idotsint\limits_{\bold x\in [0,1]^{n}}\prod_{(a, b\neq M(a))}(1-x_ax_b)\,d\bold x,\\
\mathcal P(M)&=\idotsint\limits_{\bold x\in [0,1]^{n}}\prod_{(a, b\neq M(a))}(1-x_ax_b)^2\,d\bold x.
\end{aligned}
\end{equation}
\end{Lemma}

{\bf Note.\/} The top integral in \eqref{P(Mest)=int^2'} also equals the probability that $M$ is 
stable. 

Next, introduce $R(M)$, the sum of all partners' ranks  under matching $M$.  We have
\begin{align*}
R(M)&=(n-1)+\sum_{i, j\neq M(i)}\chi\bigl(X_{i, j}< X_{i, M(i)}\bigr)\\
&=(n-1)+\sum_{(i_1,i_2)}\Bigl[\chi\bigl(X_{i_1, M(i_2)}<X_{i_1, M(i_1)}\bigr)+
\chi\bigl(X_{i_2, M(i_1)}<X_{i_2, M(i_2)}\bigr)\Bigr].
\end{align*}
\noindent For $n-1\le k,\ell \le n(n-1)$, let $\pr_{k}(M):=\!\!\pr(M\text{ is e-stable}, R(M)=k)$. 
\begin{Lemma}\label{Pk(M)'} Using notation $\bar z=1-z$,
\begin{equation}\label{Pk(M)='} 
\begin{aligned}
\pr_{k}(M)&=\idotsint\limits_{\bold x\in [0,1]^{n}}[\xi^{k-(n-1)}]\prod_{(a,b\neq M(a))}
\bigl(\bar x_a\bar x_b+\xi x_a\bar x_b+\xi \bar x_a x_b\bigr)\,d\bold x.
\end{aligned}
\end{equation}
Thus $\pr_{k}(M)$ does not depend on $M$.
\end{Lemma}

Let $M_1\neq M_2$ be two generic matchings. Together, $M_1$ and $M_2$ determine a graph $G(M_1,M_2)$ on the vertex set $[n]$, 
 with the edge set $E$ formed by the man-woman pairs $(i,j)\in M_1\cup M_2$.
Each component of $G(M_1,M_2)$ is either an edge $e\in M_1\cap M_2$,  or a circuit of even length at least $4$, in which the edges from $M_1$ and $M_2$ alternate. So the edge set for all these circuits is the symmetric difference $M_1\Delta M_2$.  Let $\mathcal N=\mathcal N(M_1,M_2)$ denote the vertex set of $M_1\Delta M_2$, and $\nu=\nu(M_1,M_2):=|\mathcal N|$. Then $I=\mathcal N^c$ is $\{i\in [n]:\, M_1(i)=
M_2(i)\}$, $|I|=n-\nu$.
\begin{Lemma}\label{P(M1,M2est)'} 
Denoting
$
\bold x_1=\{x_{i,1}:\,i\in [n]\}, \quad \bold x_2=\{x_{i, 2}:\,i\in [n]\},\quad x_{i, 1}= x_{i, 2}\,\,\text{ for }\,\,i\in I$,
and $\bold x_2^*= \{x_{i,2}:\,i\in \mathcal N\}$, we have 
\begin{equation*}
\begin{aligned}
P(M_1,M_2)&\ge\,\,\,\idotsint\limits_{\bold x_1\in [0,1]^n,\,\,\bold x_2^*\in [0,1]^{\nu}}\!\!\!\!\!\!\!\!\!
f(\bold x_1,\bold x_2)\,d\bold x_1 d\bold x_2^*,
\end{aligned}
\end{equation*}
with $f(\bold x_1,\bold x_2)$ defined in Lemma \ref{P(M1,M2est)}.
\end{Lemma}
The proofs are a shorter version of those for the two-sided matchings.

 \section{Estimates of the integrals} 
 {\bf Notations.\/} We will write $A_n\lessdot B_n$ as a shorthand for ``$A_n=O(B_n)$, uniformly over parameters that determine $A_n$, $B_n$'', when the expression for $B_n$ is uncomfortably bulky for an argument of the big-O notation. 

In \cite{Pit3} we proved that, uniformly for all matchings $M$ on $[m]$, ($m$ even),
\begin{equation}\label{simple1,2}
\prod_{(i_1,i_2\neq M(i_1))}\!\!(1-x_{i_1} x_{i_2}) \lessdot \exp\Bigl(-\frac{s^2}{2}\Bigr),\quad s:=\sum_{i\in [m]}x_i,
\end{equation}
Only minor modification is needed to prove that
\begin{equation}\label{simple1,2'}
\prod_{(i_1,i_2\neq i_1)}\!\!(1-x_{i_1} x_{i_2}) \lessdot \exp\Bigl(-\frac{s^2}{2}\Bigr),\quad s:=\sum_{i\in [m]}x_i,
\end{equation}
The bounds \eqref{simple1,2} and \eqref{simple1,2} will be instrumental in this paper as well. Another key tool is the following claim, \cite{Pit0}, \cite{Pit2}.
 \begin{Lemma}\label{intervals1} Let $X_1,\dots, X_{\nu}$ be independent $[0,1]$-Uniforms. Let 
$S=\sum_{i\in [\nu]}X_i$ and $\bold V=\{V_i=X_i/S;\, i\in [\nu]\}$, so that $\sum_{i\in [\nu]}V_i=1$.
Let $\bold L=\{L_i ;\, i\in [\nu]\}$ be the set of lengths of the $\nu$ consecutive subintervals of $[0,1]$ obtained by selecting, independently and uniformly at random, $\nu-1$ points in $[0,1]$. 
Then  the joint density $f_{S,\bold V}(s,\bold v)$, ($\bold v=(v_1,\dots,v_{\nu-1})$), of $(S,V)$ is given by
\begin{equation}\label{joint<}
\begin{aligned}
f_{S,\bold V}(s,\bold v)&=s^{\nu-1}\chi\bigl(\max_{i\in [\nu]} v_i\le s^{-1}\bigr) \chi(v_1+\cdots+v_{\nu-1}\le 1)\\
&\le \frac{s^{\nu-1}}{(\nu-1)!} f_{\bold L}(\bold v),\quad v_{\nu}:=1-\sum_{i=1}^{\nu-1}v_i;
\end{aligned}
\end{equation}
here $f_{\bold L}(\bold v)=(\nu-1)!\,\chi(v_1+\cdots+v_{\nu-1}\le 1)$ is the density of $(L_1,\dots,L_{\nu-1})$.
\end{Lemma}
We will also use the classic identities,  Andrews, Askey and Roy \cite{AndAskRoy}, Section 1.8:
\begin{equation}\label{int,prod}
\begin{aligned}
&\overbrace {\idotsint}^{\nu}_{\bold x\ge \bold 0 \atop x_1+\cdots+x_{\nu}\le 1}\prod_{i\in [\nu]} x_i^{\alpha_i}\,\,d\bold x=\frac{\prod_{i\in [\nu]}\alpha_i!}{(\nu+\alpha)!},\quad \alpha:=\sum_{i\in [\nu]}\alpha_i,\\
&\overbrace {\idotsint}^{\nu-1}_{\bold x\ge\bold 0\atop x_1+\cdots+x_{\nu}=1}\prod_{i\in [\nu]} x_i^{\alpha_i}\,\,dx_1\cdots dx_{\nu-1} =\frac{\prod_{i\in [\nu]}\alpha_i!}{(\nu-1+\alpha)!}.
\end{aligned}
\end{equation}
The identity/bound \eqref{joint<} is useful since the random vector $\bold L$ had been well studied. It is known, for instance, that 
\begin{equation}\label{L^{(nu)}=frac}
\bold L\overset{\mathcal D}\equiv\left\{\frac{W_i}{\sum_{j\in [\nu]}W_j}\right\}_{i\in [\nu]},
\end{equation}
where $W_j$ are independent, exponentially distributed, with the same parameter, 
R\'enyi \cite{Ren}. An immediate corollary is that $L_1,\dots L_{\nu}$ are equidistributed. We used \eqref{L^{(nu)}=frac}
in \cite{Pit3} to prove
\begin{Lemma}\label{sumsofLs} Let $s\ge 2$.  For $\sigma<\frac{1}{s+1}$, we have
\begin{equation}\label{sumLj^s}
\pr\Biggl(\Bigl|\frac{\nu^{s-1}}{s!}\sum_{j\in [\nu]}L_j^s-1\Bigr|\ge \nu^{-\sigma}\Biggr)
\le\exp\bigl(-\Theta(\nu^{\frac{1}{s+1}-\sigma}\,)\bigr),
\end{equation}
and, for $\nu$ even,
\begin{equation}\label{sumL_jL_{j+}}
\pr\Biggl(\Bigl|2\nu\sum_{j\in [\nu/2]}L_jL_{j+\nu/2}-1\Bigr|\ge \nu^{-\sigma}\Biggr)
\le\exp\bigl(-\Theta(\nu^{\frac{1}{s+1}-\sigma}\,)\bigr).
\end{equation}
\end{Lemma}

{\bf Note.\/} Had $\ex\bigl[e^{zW^2}\bigr]$ been finite for $|z|\neq 0$ sufficiently small, we would have been
 able to prove--via a standard application of Chernoff's method--a stronger bound, namely $\exp\bigl(-\Theta(\nu)\bigr)$. And that's the estimate we claimed (without a proof) in
 \cite{Pit1}, overlooking that, for the exponential $W$,  $\ex\bigl[e^{zW^2}\bigr]=\infty$ if $z>0$. The weaker, sub-exponential, bounds \eqref{sumLj^s}, \eqref{sumL_jL_{j+}} were proved in \cite{Pit3} by combining Chernoff's method with truncation
of $W$  at $\nu^{\frac{1}{s+1}}$. It turned out that these bounds combined  with the inequality \eqref{simple1,2}, missed in \cite{Pit1},
were all we needed in \cite{Pit3} 
for the asymptotic study of {\it non-bipartite\/} stable partitions and matchings.  We stressed there that the argument could be used
as a template for swapping some proof steps in the corresponding parts of \cite{Pit1}, \cite{Pit2},
\cite{IrvPit}, and so avoiding the problematic issue of exponential bounds. We will see that the sub-exponential bounds 
\eqref{sumLj^s}, \eqref{sumL_jL_{j+}} are sufficient for our study in this paper as well.\\

In addition to the bounds \eqref{sumLj^s}, we will need
\begin{equation}\label{P(L^+>)<}
\pr\left(\max_{j\in [\nu]} L_j^{(\nu)}\ge \frac{\log ^2\nu}{\nu}\right)\le e^{-\Theta(\log^2\nu)},
\end{equation}
which directly follows from 
\begin{equation}\label{P(maxL_j>)}
\pr\left(\max_{j\in [\nu]} L_j^{(\nu)}\ge x\right)\le \nu\! \pr\bigl(L_1^{(\nu)}\ge x\bigr) =\nu (1-x)^{\nu-1}. 
\end{equation}

\section{Estimates for two-sided matchings} 
\subsection{$\pr(M)$, $\ex[S_n]$} By Lemma \ref{P(Mest)=}, 
\begin{equation}\label{P(M)=(Int)^2}
P(M)=\pr(M\text{ is e-stable})\equiv\left(\idotsint\limits_{\bold x\in [0,1]^n}\prod_{(a,b)}(1-x_ax_b)\,d\bold x\right)^2.
\end{equation}
Here, by \eqref{simple1,2'},  
\[
\prod_{(a,b)}(1-x_ax_b)\lessdot \exp\Bigl(-\frac{s^2}{2}\Bigr),\quad s:=\sum_{a\in [n]}x_a.
\] 
So, by Lemma \ref{intervals1} and \eqref{int,prod},
\begin{equation}\label{prelim}
\begin{aligned}
\idotsint\limits_{\bold x\in [0,1]^n}\,\prod_{(a,b)}&(1-x_ax_b)\,d\bold x\lessdot
\idotsint\limits_{\bold x\in [0,1]^n}\exp\Bigl(-\frac{s^2}{2}\Bigr)\,d\bold x\\
&\le\int_0^{\infty}\exp\Bigl(-\frac{s^2}{2}\Bigr)\frac{s^{n-1}}{(n-1)!}\,ds\lessdot\frac{(n-2)!!}{(n-1)!}
=\frac{1}{(n-1)!!}.
\end{aligned}
\end{equation}
Therefore $\pr(M)\lessdot \bigl[(n-1)!!\bigr]^{-2}$, implying that
\[
\ex\bigl[S_n\bigr]=n!\!\pr(M)\lessdot\frac{n!}{\bigl[(n-1)!!\bigr]^2}=\frac{n!!}{(n-1)!!}=\Theta(n^{1/2}).
\]
This bound is qualitatively sharp.
\begin{Theorem}\label{ES_nsim}
\[
\ex\bigl[S_n\bigr]=\bigl(1+O(n^{-\sigma})\bigr)\sqrt{\frac{\pi n}{2}},\qquad\forall\,\sigma<\frac{1}{3}.
\]
\end{Theorem}
\noindent The proof consists of two parts: reduction of the integration domain in the formula \eqref{P(M)=(Int)^2}
and sharp estimate of the integral over the core domain.

{\bf Note.\/} For comparison: (1) the expected number of the classical, bipartite, stable matchings is asymptotic
to $e^{-1} n\log n\gg n^{1/2}$, \cite{Pit0}; (2) its counterpart for one-sided stable matchings approaches
a finite limit $e^{1/2}$ \cite{Pit2}, and even the expected number of stable {\it partitions\/}, that include stable
matchings as a very special case (Tan \cite{Tan}), is of order $n^{1/4}$ \cite{Pit3}, again well below $n^{1/2}$.

\subsubsection{Reduction of the cube $[0,1]^n$} In steps, we will eliminate the large chunks of the integration
cube so that we will be able to approximate sharply the integrand on the remaining part of the cube,
at the total (relative) error cost of order $e^{-\Theta(\log^2 n)}$. Since the argument is very close to, indeed
simpler than, the proof in Section 4.2 in \cite{Pit3}, we limit ourselves to describing the intermediate
steps and shedding some light on the proofs.

For the first reduction, we observe that the integrand $e^{-\frac{s^2}{2}} s^{n-1}$ in \eqref{prelim} attains its
sharply pronounced maximum at $(n-1)^{1/2}$, and the second order logarithmic derivative of the integrand
is below $-1$ for all $s>0$. Given $C\subseteq [0,1]^n$, define
\[
I_C(M)=\idotsint\limits_{\bold x\in C}  \prod_{(a,b)}(1-x_ax_b)\,d\bold x,
\]
and set $I(M):=I_{[0,1]^n}(M)$.
\begin{Lemma}\label{C1} Let $C_1=\bigl\{\bold x\in [0,1]^n:\,s\le s_n\bigr\}$, $s_n=n^{1/2}+3 \log n$. Then,
 \[
 I(M)-I_{C_1}(M)\le \frac{e^{-\Theta(\log^2 n)}}{(n-1)!!}.
 \]
 \end{Lemma}
\noindent See Lemma 4.5 in \cite{Pit3}. Our next step is to shrink $C_1$ to 
its subset where each component $x_i$ of $\bold x$ is at most $s\frac{\log^2 n}{n}$.
\begin{Lemma}\label{C2} Let $u_i:=\frac{x_i}{s}$ and $C_2:=\Bigl\{\bold x\in C_1\,:\,\max_i u_i\le\frac{\log^2 n}{n}\Bigr\}$. Then
\[
I_{C_1}(M)-I_{C_2}(M)\le \frac{e^{-\Theta(\log^2 n)}}{(n-1)!!}.
\]
\end{Lemma}
\begin{proof} The proof starts with
\begin{align*}
I_{C_1}(M)-I_{C_2}(M)&\lessdot\idotsint\limits_{\bold x\ge 0}e^{-\frac{s^2}{2}}\,\chi\Bigl\{\max_i u_i\ge \frac{\log^2 n}{n}\Bigr\}\,d\bold x\\
&\le\frac{\pr\Bigl(\max_{i\in [n]}L_i\ge \frac{\log^2n}{n}\Bigr)}{(n-1)!}\int_0^{\infty}e^{-\frac{s^2}{2}}s^{n-1}\,ds,
\end{align*}
see Lemma \ref{intervals1}, and the probability is then bounded with the help of \eqref{P(L^+>)<}.
\end{proof}

Notice that $\sum_{i\in [n]} x_i^4=O\bigl(n^{-1}\log^8 n)$ uniformly for $\bold x\in C_2$, which would have
been good enough for us. However, the constraints on $C_2$ guarantee only that
 $\sum_{i\in [n]} x_i^2=O\bigl(\log^4 n)$, while we will need to know that only $\bold x$ with a bounded
 $\sum_{i\in [n]}x_i^2$ matter asymptotically. Thus another reduction of the integration domain is in
 order.

\begin{Lemma}\label{C3}  Let  $C_3:=\!\Bigl\{\bold x\in C_2\,: \Bigl|\frac{n}{2}\sum_i u_i^2-1\Bigr|\le n^{-\sigma}\!\Bigr\}$, \,$\sigma<1/3$. Then
\[
I_{C_3}(M)-I_{C_2}(M)\le \frac{e^{-\Theta(n^{1/3-\sigma})}}{(n-1)!!}.
\]
\end{Lemma}
\noindent The proof combines Lemma \ref{intervals1} and Lemma \ref{sumsofLs}.\\

Putting together Lemmas \ref{C1}, \ref{C2} and \ref{C3} we have

\begin{Corollary}\label{C3,expl}
\[
I(M)-I_{C_3}(M)\le \frac{e^{-\Theta(\log^2 n)}}{(n-1)!!},
\]
where the core domain $C_3\subset [0,1]^n$ is defined by the constraints: with $s_n=n^{1/2}+3\log n$,
\begin{equation}\label{C3def}
s \le s_n\quad \max_{i\in [n]} x_i\le s\frac{\log^2 n}{n},
\quad \left|\frac{n\sum_{i\in [n]}x_i^2}{2s^2}-1\right|\le n^{-\sigma}.
\end{equation}
\end{Corollary}

Notice that the constraints \eqref{C3def} imply that 
\begin{equation}\label{max x<}
\max_{I\in [n]} x_i\le 2n^{-1/2}\log^2 n\to 0,
\end{equation}
meaning that the constraint $\max_i x_i\le 1$ is superfluous for $n$ large enough. Furthermore,
in combination with $1-z=\exp\bigl[-z-z^2/2+O(|z|^3)\bigr]$, $z\to 0$, the inequality \eqref{max x<}
delivers 
\[
\prod_{(i, j)} (1-x_i x_j)=\exp\Biggl(-\sum_{(i, j)}\Bigl(x_ix_j+\frac{x_i^2x_j^2}{2}\Bigr)+O\Bigl(\sum_{i\in [n]} x_i^4\Bigr)\Biggr),
\]
the equality that holds uniformly for $\bold x\in C_3$, thus with the remainder term of order
$O\bigl(n^{-1}\log^8 n\bigr)$. From the constraints \eqref{C3def} we infer
\begin{Lemma}\label{prodsim} Uniformly for $x\in C_3$,
\[
\prod_{(i, j)} (1-x_i x_j)=\exp\!\left(\!-\frac{s^2}{2}\left(\!1-\frac{2}{n}\!\right)-\frac{s^4}{n^2}+O(n^{-\sigma})\!\right).
\]
\end{Lemma}
\subsubsection{Sharp estimate of $P(M)$} 
\begin{Lemma}\label{IC3sim}
\[
I_{C_3}(M)=c_n\frac{1-e^{-\Theta(\log^2 n)}}{(n-1)!!}, \quad c_n=\left\{\begin{aligned}
&1,&&n\text{ is even},\\
&\sqrt{\frac{\pi}{2}},&&n\text{ is odd}.\end{aligned}\right. 
\]
\end{Lemma}
\begin{proof}
Denote $\psi_n(s)=\frac{s^2}{2}\bigl(1-\frac{2}{n}\bigr)-\frac{s^4}{n^2}$.
Applying Lemma \ref{intervals1} and using \eqref{max x<}, Lemma \ref{prodsim}, we obtain
\begin{align*}
I_{C_3}(M)&=\bigl(1+O(n^{-\sigma})\bigr)\pr\Bigl(\Bigl|\frac{n}{2}\!\sum_{i\in [n]}L_i^2-1\Bigr|\le n^{-\sigma};\,\, \max_{j\in [n]} L_j\le \frac{\log^2 n}{n}\Bigr)\\
&\quad\times\frac{1}{(n-1)!}\int_0^{s_n} e^{-\psi_n(s)} s^{n-1}\,d s.
\end{align*}
The probability factor here exceeds $1-\exp(-\Theta(\log^2 n))$. The integrand attains its sharp maximum
at $\hat s=(n-1)^{1/2} -\Theta(n^{-1/2})$, so that $s_n-\hat s\ge 2\log n$. The overwhelming contribution
to the integral comes from $s\in [\hat s -\log n,\hat s+\log n]$, and for those $s$ we have
$\frac{s^4}{n^2}=1+O(n^{-1/2}\log n)$. An easy argument shows then that the integral equals
\begin{align*}
&\frac{e^{-1}\bigl(1+O(n^{-1/2}\log n)\bigr)}{(n-1)!}\int_0^{\infty}\!\! e^{-\frac{s^2}{2}\left(1-\frac{2}{n}\right)} s^{n-1}\,ds\\
&\quad =\bigl(1+O(n^{-1/2}\log n)\bigr)\cdot \frac{c_n}{(n-1)!!}.
\end{align*}
\end{proof}
\begin{Corollary}\label{inprod=sharp}
\begin{align*}
\pr(M)&=\left(\idotsint\limits_{\bold x\in [0,1]^n}  \prod_{(a,b)}(1-x_ax_b)\,d\bold x\right)^2
=\bigl(1+O(n^{-\sigma})\bigr)\left(\frac{c_n}{(n-1)!!}\right)^2.
\end{align*}
\end{Corollary}
By $\ex\bigl[S_n\bigr]=n! P(M)$, and Stirling formula for factorials, Corollary \ref{inprod=sharp}
completes the proof of Theorem \ref{ES_nsim}.

\subsection{Likely range of the partners' ranks} Let $R_{w}(M)$ and $R_{m}(M)$ stand for the total wives' rank and
the total husbands' rank in a generic matching $M$. From \eqref{Pkell(M)=}, $R_{w}(M)$ and $R_{m}(M)$ are equidistributed and, for $R(M)=R_m(M),\,R_w(M)$, $\pr_k(M):=\pr(M\text{ is e-stable},\,R(M)=k)$ is given by
\begin{equation}\label{Pk(M)=} 
\begin{aligned}
\pr_{k}(M)&=\idotsint\limits_{\bold x\in [0,1]^{n}}[\xi^{k-n}]\prod_{(a,b)}
\bigl(\bar x_a\bar x_b+\xi x_a\bar x_b+\xi \bar x_a x_b\bigr)\,d\bold x\\
&\times\idotsint\limits_{\bold y\in [0,1]^{n}} \prod_{(c,d)}
\bigl(1-y_cy_d\bigr)\,d\bold y.
\end{aligned}
\end{equation}
\begin{Theorem}\label{R(M)appr} For a fixed $\eps\in (0,1)$,
\[
\pr\left(\max_M\left|\frac{R(M)}{n^{3/2}}-1\right|\ge \eps\right)\le e^{-\Theta(\log^2 n)}.
\]
\end{Theorem}
\begin{proof} Introduce $k=\lceil (1+\eps)n^{3/2}\rceil$, and define 
\[
P^+(M)=\pr(M\text{ is e-stable},\,R(M)\ge k). 
\]
Applying Chernoff's method to \eqref{Pk(M)=}, we get 
\begin{align*}
P^+(M)&\le I(k) \idotsint\limits_{\bold y\in [0,1]^{n}}\prod_{(c,d)}\bigl(1-y_cy_d\bigr)\,d\bold y,\\
I(k)&:=\idotsint\limits_{\bold x\in [0,1]^{n}}\inf_{\xi\ge 1}\Biggl[\xi^{-\bar k}\prod_{(a,b)}
\bigl(\bar x_a\bar x_b+\xi x_a\bar x_b+\xi \bar x_a x_b\bigr)\Biggr]\,d\bold x,
\end{align*}
$\bar k:=k-n$. By \eqref{prelim}, the first line integral is of order $\frac{1}{(n-1)!!}$. As for $I(k)$,
in Theorem 4, \cite{Pit2} (stable matchings on $[n]$, $n$ even), and recently in Theorem 4.16, \cite{Pit3} (stable partitions on $[n]$) we analyzed similar integrals, where the products were over the unmatched (unaligned) pairs, while in the present case the product is over all pairs $(a, b)$ of distinct elements $a,\,b
\in [n]$. Since in
all cases the number of excluded pairs is linear in $n$, the arguments from the cited papers work just as well
for $I(k)$, and we get
\[
I(k)\le \frac{e^{-\Theta(\log^2 n)}}{(n-1)!!}\Longrightarrow
P^+(M)\le  \frac{e^{-\Theta(\log^2 n)}}{[(n-1)!!]^2}.
\]
Likewise
\[
P^-(M):=\pr\bigl(M\text{ is e-stable},\, R(M)\le (1-\eps)n^{3/2}\bigr) \le \frac{e^{-\Theta(\log^2 n)}}{[(n-1)!!]^2}.
\]
Therefore
\begin{align*}
\pr\left(\max_M\left|\frac{R(M)}{n^{3/2}}-1\right|\ge \eps\right)&\le n!\,\frac{e^{-\Theta(\log^2 n)}}{[(n-1)!!]^2}
\le n^{1/2}e^{-\Theta(\log^2 n)},
\end{align*}
which completes the proof of Theorem \ref{R(M)appr}.
\end{proof}

\subsection{A doubly stable matching is unlikely} Our task is to prove that $\sum_M\mathcal P(M)\to 0$.
To bound $\mathcal P(M)$ we need first to reduce the
integration domain $\{\bold x, \bold y\in [0,1]^n\}$ to a manageable subdomain $D^*$,
such that 
\[
n!\max_M\bigl[\mathcal P(M)-\mathcal P_{D^*}(M)\bigr]\to 0. 
\]
By \eqref{mathcal P(M|)<P(M|)} and the inequality \eqref{simple1,2'}, we have
\begin{equation*}
\mathcal P(M|\bold x,\bold y)\lessdot\exp\left(-\frac{s^2}{2}-\frac{t^2}{2}\right),\quad s:=\sum_i x_i,\,\,t:=\sum_i y_i.
\end{equation*}
Given $D$, a subset of $\{\bold x,\bold y\in [0,1]^n\}$, denote
\[
\mathcal P_{D}(M)=\idotsint\limits_{\bold x,\,\bold y\in D}\mathcal P(M |\bold x,\bold y)\,d\bold xd\bold y.
\]
Then
\[
\mathcal P_{D}(M)\lessdot \idotsint\limits_{\bold x,\,\bold y\in D}\exp\left(-\frac{s^2}{2}-\frac{t^2}{2}\right)\,d\bold x d\bold y.
\]
Let $D_1=\{\bold x, \bold y: \max(s,t) \le 2n^{1/2}\}$. Then
\begin{align*}
\mathcal P(M)-\mathcal P_{D_1}(M)&\lessdot \idotsint\limits_{\bold x\ge \bold 0\atop s\ge 2n^{1/2}}\exp\left(\!-\frac{s^2}{2}\right)\,d\bold x\,\cdot \,\idotsint\limits_{\bold y\ge \bold 0}\exp\left(\!-\frac{t^2}{2}\right)\,d\bold y.
\end{align*}
The second integral equals $1/(n-1)!!$, and the first integral is of order
\begin{align*}
ne^{-2n}\frac{(2n^{1/2})^{n-1}}{(n-1)!}\lessdot \frac{n (2e^{-2})^n}{(n-1)!!},
\end{align*}
since the integrand attains its maximum at $s=2n^{1/2}$. So
\begin{equation}\label{D1}
\mathcal P(M)-\mathcal P_{D_1}(M)\lessdot \frac{n(2e^{-2})^n}{\bigl[(n-1)!!\bigr]^2}.
\end{equation}
For the second, last, reduction, define $u_i=\frac{x_i}{s}$, $v_i=\frac{y_i}{t}$ and set 
\[
D_2=\Bigl\{(\bold x,\bold y)\in D_1: \max_i u_i\le n^{-\gamma},\,\max_i v_i\le n^{-\gamma}
\Bigr\},
\]
where $\gamma<1$ is to be chosen later. Then
\begin{align*}
\mathcal P_{D_1}(M)-\mathcal P_{D_2}(M)&\lessdot \idotsint\limits_{\bold x\ge \bold 0\atop \max u_i>n^{-\gamma}}\exp\left(\!-\frac{s^2}{2}\right)\,d\bold x\,\cdot\, 
\idotsint\limits_{\bold y\ge \bold 0}\exp\left(\!-\frac{t^2}{2}\right)\,d\bold y.
\end{align*}
By Lemma \ref{intervals1} and \eqref{P(maxL_j>)}, the first integral is bounded by 
\[
\frac{\pr\Bigl(\max_{i\in [n]}L_i\ge n^{-\gamma}\Bigr)}{(n-1)!}\int_0^{\infty}e^{-\frac{s^2}{2}}s^{n-1}\,ds
\le \frac{e^{-\Theta(n^{1-\gamma})}}{(n-1)!!}.
\]
So
\begin{equation}\label{D2}
\mathcal P_{D_1}(M)-\mathcal P_{D_2}(M)\le \frac{e^{-\Theta(n^{1-\gamma})}}{\bigl[(n-1)!!\bigr]^2}.
\end{equation}
Setting $D^*=D_2$, by \eqref{D1} and \eqref{D2} we have
\begin{equation}\label{D*}
\mathcal P(M)-\mathcal P_{D^*}(M)\le \frac{e^{-\Theta(n^{1-\gamma})}}{\bigl[(n-1)!!\bigr]^2}.
\end{equation}
Now that $(\bold x,\bold y)\in D^*$ we can use \eqref{mathcalP(M|)=} to obtain a sharp upper bound for $\mathcal P(M|\bold x,\bold y)$, whence for $\mathcal P_{D^*}(M)$ the integral of
$\mathcal P(M|\bold x,\bold y)$ over $D^*$. First of all, on $D^*$ we have $x_i,\,y_i
\le 2n^{1/2-\gamma}\to 0$, provided that $\gamma\in (1/2,1)$.

By Bonferroni inequality
\[
\pr\bigl(\cap B_j^c\bigr)\le 1-\sum_j\pr(B_j)+\sum_{j_1<j_2}\pr\bigl(B_{j_1}\cap B_{j_2}),
\]
the $(i_1,i_2)$-th factor from the product in \eqref{mathcalP(M|)=} is at most
\begin{equation}\label{G(i_1,i_2)=}
\begin{aligned}
F_{(i_1,i_2)}(\bold x,\bold y)&=1-G_{(i_1,i_2)}(\bold x,\bold y)+H_{(i_1,i_2)}(\bold x,\bold y),\\
G_{(i_1,i_2)}(\bold x,\bold y)&:=x_{i_1}y_{M(i_2)}+x_{i_2}y_{M(i_1)}+x_{i_1}x_{i_2}+
y_{M(i_2)}y_{M(i_1)},\\
H_{(i_1,i_2)}(\bold x,\bold y)&:=2x_{i_1}x_{i_2}y_{M(i_1)}y_{M(i_2)}+x_{i_1}x_{i_2}\bigl(y_{M(i_1)}+y_{M(i_2)}\bigr)\\
&\quad+y_{M(i_1)}y_{M(i_2)}\bigl(x_{i_1}+x_{i_2}\bigr).
\end{aligned}
\end{equation}
Here $G_{(i_1,i_2)}(\bold x,\bold y),\,H_{(i_1,i_2)}(\bold x,\bold y)\to 0$ uniformly for all $i_1,\,i_2$ 
and $(\bold x,\bold y)\in D^*$, and more precisely
$H_{(i_1,i_2)}(\bold x,\bold y)=O(n^{3/2-3\gamma})$. So
\begin{align*}
F_{(i_1,i_2)}(\bold x,\bold y)&\le \bigl(1-G_{(i_1,i_2)}(\bold x,\bold y)\bigr) e^{O(n^{3/2-3\gamma})}\\
&\le \bigl(1-x_{i_1}x_{i_2}\bigr) \bigl(1-y_{M(i_1)}y_{M(i_2)}\bigr)\\
&\quad\times\bigl(1-x_{i_1}y_{M(i_2)}\bigr)
\bigl(1-x_{i_2}y_{M(i_1)}\bigr) e^{O(n^{3/2-3\gamma})}.
\end{align*}
Now 
\begin{align*}
&\qquad\qquad \prod_{(i_1,i_2)}\!\bigl(1-x_{i_1}x_{i_2}\bigr)\lessdot e^{-\frac{s^2}{2}},\quad \prod_{(i_1,i_2)}\!\bigl(1-y_{M(i_1)}y_{M(i_2)}\bigr)\lessdot e^{-\frac{t^2}{2}},\\
&\prod_{(i_1,i_2)}\!\bigl(1-x_{i_1}y_{M(i_2)}\bigr)\bigl(1-x_{i_2}y_{M(i_1)}\bigr) 
\le\exp\Bigl(-\!\sum_{(i_1,i_2)}\!(x_{i_1}y_{M(i_2)}+x_{i_2}y_{M(i_1)})\!\Bigr)\\
&\qquad\qquad\qquad\quad=\exp\Bigl(-st+\sum_ix_iy_{M(i)}\Bigr)\le e^{-st} e^{O(n^{2-2\gamma})}.
\end{align*}
Therefore, uniformly for $(\bold x,\bold y)\in D^*$, we have
\[
\mathcal P(M|\bold x,\bold y )\le e^{-\frac{\xi^2}{2}}\cdot e^{O(n^{7/2-3\gamma})}, \quad
\xi:=s+t.
\]
Consequently
\begin{align*}
\mathcal P_{D^*}(M)&\le e^{O(n^{7/2-3\gamma})}\idotsint\limits_{\bold x,\, \bold y\ge\bold 0}
e^{-\frac{\xi^2}{2}}\,d\bold x d\bold y\\
&=e^{O(n^{7/2-3\gamma})}\int\limits_0^{\infty} e^{-\frac{\xi^2}{2}} \frac{\xi^{2n-1}}{(2n-1)!}\,d\xi
=\frac{e^{O(n^{7/2-3\gamma})}}{(2n-1)!!}.
\end{align*}
Combining this estimate with \eqref{D*} we conclude that
\[
\mathcal P(M)\le \frac{e^{O(n^{7/2-3\gamma})}}{(2n-1)!!}+\frac{e^{-\Theta(n^{1-\gamma})}}{\bigl[(n-1)!!\bigr]^2},
\]
uniformly for all $M$. So
\begin{align*}
\sum_M \mathcal P(M)&\le \frac{e^{O(n^{7/2-3\gamma})}n!}{(2n-1)!!} +\frac{e^{-\Theta(n^{1-\gamma})}n!}{\bigl[(n-1)!!\bigr]^2}\\
&\le \exp\bigl(-n\log 2+O(n^{7/2-3\gamma})\bigr)+e^{-\Theta(n^{1-\gamma})}\\
&=e^{-\Theta(n^{1-\gamma})},
\end{align*}
provided that $\gamma\in (5/6,1)$. Thus we have proved
\begin{Theorem}\label{P(Mdeexists)} The probability that there exists a matching $M$, which is both e-stable and stable, is at most $e^{-n^{\sigma}}$ for every $\sigma<1/6$.
\end{Theorem}

\subsection{$\ex\bigl[S_n^2\bigr]$ and such} Having proved that $\ex\bigl[S_n\bigr]$ is of order $n^{1/2}$, we
felt confident that--- like other types of stable matchings we studied earlier---the second moment $\ex\bigl[S_n^2\bigr]$ would not grow faster than $n^{\gamma}$, for some $\gamma\ge 1$. Contrary to our
naive expectations, $\ex\bigl[S_n^2\bigr]$ grows much faster.

For $\xi\in (0,1)$, define
\begin{equation}\label{H(xi)=}
\begin{aligned}
\mathcal H(\xi)&= -(1-\xi)\log(1-\xi)+4\xi \log\frac{1+\xi+\sqrt{1-2\xi+5\xi^2}}{1-\xi+\sqrt{1-2\xi+5\xi^2}}\\
&\quad-(1+\xi)\log\frac{1+3\xi^2+(\xi+1)\sqrt{1-2\xi+5\xi^2}}{1-\xi+\sqrt{1-2\xi+5\xi^2}}.
\end{aligned}
\end{equation}
$\mathcal H(0+)=\mathcal H(1-)=0$, and $\mathcal H(\xi)$ attains its maximum at $\xi_{\text{max}}\approx 0.739534$, with
$H(\xi_{\text{max}})\approx 0.253062$.

Using $A\gtrdot B$ as a shorthand for $B=O(A)$, we have
\begin{Theorem}\label{ESn^2>} $\ex\bigl[S_n^2\bigr]\gtrdot n^{3/2}\exp[nH(\xi_{\text{max}})]> 
n^{3/2}1.28^n$,
for $n$ large enough.
Consequently, for each such $n$ there exists an instance of the $2n$ preference lists with the number of e-stable matchings exceeding $n^{3/4}1.28^{n/2}> n^{3/4}1.13^n$.
\end{Theorem}
Thus the standard deviation of $S_n$ is more than $1.13^n$, dwarfing $\ex\bigl[S_n\bigr]$. Informally,
the distribution of $S_n$ is highly asymmetric, with the discernible right tail much longer than the left tail. It is tempting to conjecture that $\pr(S_n>0)\to 1$. 

To begin the proof, we observe that 
\begin{equation}\label{sum>lower}
\ex\bigl[(S_n)_2\bigr]=\sum_{M_1\neq M_2}\pr(M_1,M_2),
\end{equation}
where $\pr(M_1,M_2)$ is the probability that both $M_1$ and $M_2$ are e-stable. The lower bound
for $\pr(M_1,M_2)$ given in Lemma \ref{P(M1,M2est)} depends on $M_1$, $M_2$ only through
$2\nu:=2\nu(M_1,M_2)$ the total length of the bipartite circuits formed by the alternating pairs (man,woman) matched in either $M_1$ or, exclusively, in $M_2$. It makes sense to guess that the dominant contribution to the resulting lower bound for the sum in \eqref{sum>lower} comes from the pairs $(M_1,M_2)$ with
$\nu(M_1,M_2)$ relatively close to some judiciously chosen $\nu$. And since we are after 
an exponential bound, we will use a single--$\nu$ bound coming from Lemma \ref{P(M1,M2est)}: with $\bold x_1,\,\bold x_2\in [0,1]^n$, and $\bold x_2^*$ formed by the first $\nu$ components of $\bold x_2$,
\begin{equation}\label{single}
\begin{aligned}
\ex\bigl[(S_n)_2\bigr] &\ge B(n,\nu)
\left(\,\,\,\idotsint\limits_{\bold x_1\in [0,1]^n,\,\,\bold x_2^*\in [0,1]^{\nu}}\!\!\!\!\!\!\!\!\!
f(\bold x_1,\bold x_2)\,d\bold x_1 d\bold x_2^*\right)^2,\\
f(\bold x_1,\bold x_2)&=\prod_{(i_1,i_2):\,i_1,i_2\in [\nu]}\!\!\!\!\bigl(1-x_{i_1, 1}\,x_{i_2, 1}\bigr)\bigl(1-x_{i_1, 2}\,x_{i_2, 2}\bigr)\\
&\quad\times\!\!\prod_{i_1\in [\nu]^c,\, i_2\in [\nu]}\bigl(1-x_{i_1, 1}\,x_{i_2, 1}\bigr)\bigl(1-x_{i_1, 2}\,x_{i_2, 2}\bigr)\\
&\quad\times\!\!\!\!\!\prod_{(i_1,i_2):\,i_1,i_2\in [\nu]^c}\!\!\!\bigl(1-x_{i_1,1}x_{i_2,1}\bigr),
\end{aligned}
\end{equation}
where  $B(n,\nu)$ is the total number of pairs $(M_1,M_2)$ of general matchings $M_1$ and $M_2$, with 
$2\nu(M_1,M_2)=2\nu$. More explicitly, we have $B(n,\nu)=\binom{n}{\nu}^2 (n-\nu)! B(\nu)$.
Here $B(\nu)$ is the total number of the disjoint unions of bipartite {\it circuits\/} on the vertex set $[\nu]\cup [\nu]$, with every second edge on each circuit marked as belonging to the matching $M_1$, and the intervening edges being assigned to the matching $M_2$. Thus $B(\nu)$ is also the total number of bipartite permutations
of $[\nu]\cup [\nu]$ with {\it cycles\/} of length $\ge 4$. A simple bijective argument shows that
$B(\nu)=\nu!\pi(\nu)$, where $\pi(\nu)$ is the total number of permutations of $[\nu]$ without a fixed point.
Since $\pi(\nu)\sim e^{-1}\nu!$,  as $\nu\to\infty$, it follows that $B(\nu)=\Theta\bigl((\nu!)^2\bigr)$. 
Therefore
\begin{equation}\label{single,simple}
\ex\bigl[(S_n)_2\bigr] \gtrdot\frac{(n!)^2}{(n-\nu)!}
\left(\,\,\,\idotsint\limits_{\bold x_1\in [0,1]^n,\,\,\bold x_2^*\in [0,1]^{\nu}}\!\!\!\!\!\!\!\!\!
f(\bold x_1,\bold x_2)\,d\bold x_1 d\bold x_2^*\right)^2.
\end{equation}
It remains to find a lower, $\nu$-dependent, bound for the multidimensional integral, simple enough
 to identify a value $\nu=\nu(n)$ that makes the resulting bound fast approach infinity. We focus
on the case when $\nu$ and $n-\nu$ are both of order $\Theta(n)$.

Similarly to the case of $\ex[S_n]$, the rest of the proof has two components: determination of the
potentially dominant core $\mathcal C$ of the integration
domain in \eqref{single,simple} and a sufficiently sharp, lower, bound of the integral over $\mathcal C$.

Motivated by our analysis of $\ex[S_n]$, and by Corollary \ref{C3,expl} in particular, we define $\mathcal C$
as follows. Define $\mathcal I_1=\mathcal I_2=[\nu]$, $\mathcal I_3=[n]\setminus [\nu]$, and
$x_{i,3}=x_{i,1} (=x_{i,2})$ for $i\in \mathcal I_3$. Denoting
\begin{align*}
&s_t=\sum_{i\in \mathcal I_t}x_{i,t},\quad s_t^{(2)}=\sum_{i\in \mathcal I_t} x_{i,t}^2, \quad s=\sum_ts_t,
\end{align*}
$\mathcal C$ is the set of all $(\bold x_1,\bold x_2)$ such that
\begin{equation}\label{mathC}
\max_{i\in \mathcal I_t} x_{i,t}\le s_t\frac{\log^2 n}{n},\quad \frac{s_t^{(2)}}{s_t^2}\le \frac{3}{|\mathcal I_t|},
\quad s=\Theta(n^{1/2}).
\end{equation}
The definition of the range of $s$ will be specified shortly. Let $(\bold x_1,\bold x_2^*)\in \mathcal C$. For large $n$, we have $\mathcal C\subset [0,1]^n \times [0,1]^{\nu}$ since $x_{i,t}=O\bigl(n^{-1/2}\log n\bigr)$ on $\mathcal C$. This bound on $x_{i,t}$  yields
 \begin{align*}
\log\!\!\!\!\!\prod_{(i_1,i_2):\,i_1,i_2\in \mathcal I_t}\!\!\!\!\!\!\bigl(1-x_{i_1, t}\,x_{i_2, t}\bigr)&\ge
-\sum_{(i_1,i_2):\,i_1,i_2\in \mathcal I_t}\!\!\!\bigl(x_{i_1, t}\,x_{i_2, t}+x_{i_1, t}^2\,x_{i_2, t}^2\bigr)\\
&\ge -\frac{s_t^2}{2}-\frac{\bigl(s_t^{(2)}\bigr)^2}{2},\qquad x\in \mathcal C.
\end{align*}
Similarly, for $t=1,2$,
\begin{align*}
\log\!\!\prod_{i_1\in \mathcal I_3,\, i_2\in\mathcal I_t}\!\!\bigl(1-x_{i_1, t}\,x_{i_2, t}\bigr)&\ge 
-\sum_{i_1\in \mathcal I_3,\, i_2\in\mathcal I_t}\!\!\bigl(x_{i_1, t}\,x_{i_2, t}+x_{i_1, t}^2\,x_{i_2, t}^2\bigr)\\
&=-s_t s_3 -s_t^{(2)} s_3^{(2)}.
\end{align*}
It follows then from \eqref{single} that, with $s=\sum_ts_t$,
\begin{align*}
&\log f(\bold x_1,\bold x_2)\ge -\frac{\sum_t s_t^2}{2}-s_1s_3-s_2s_3-1.5\sum_t \bigl(s_t^{(2)}\bigr)^2\\
=&-\frac{s^2}{2}+s_1s_2- \sum_t \bigl(s_t^{(2)}\bigr)^2\ge -\frac{s^2}{2}+s_1s_2-4.5s^2\sum_t\frac{1}
{|\mathcal I_t|}\\
&\qquad\qquad\quad\,\,=-\frac{s^2}{2}+s_1s_2 +O(1),
\end{align*}
uniformly for $n$ and $(\bold x_1,\bold x_2^*)\in \mathcal C$. Therefore, for  all positive integers $k$,
\begin{align*}
\idotsint\limits_{\bold x_1\in [0,1]^n,\,\,\bold x_2^*\in [0,1]^{\nu}}\!\!\!\!\!\!\!\!\!
f(\bold x_1,\bold x_2)\,d\bold x_1 d\bold x_2^*&\gtrdot\idotsint\limits_{(\bold x_1,\bold x_2^*)\in
\mathcal C}\exp\Bigl(-\frac{s^2}{2}+s_1s_2\Bigr)\,d\bold x_1 d\bold x_2^*\\
&\ge\frac{1}{k!}\idotsint\limits_{(\bold x_1,\bold x_2^*)\in
\mathcal C}\exp\Bigl(-\frac{s^2}{2}\Bigr)s_1^ks_2^k\,d\bold x_1 d\bold x_2^*.
\end{align*}
We will use this bound for $k=\Theta(n)$. Let $\{L_i^{(t)}\}_{i\in \mathcal I_t}$, ($t=1,2,3$), denote the lengths of $|\mathcal I_t|$ consecutive intervals obtained
by sampling $|\mathcal I_t|$  points  uniformly at random, and independently, from the interval $[0,1]$. 
(The three sampling procedures are implemented independently of each other.) Applying Lemma \ref{intervals1}, we obtain
\begin{align*}
&\qquad\quad\,\,\idotsint\limits_{(\bold x_1,\bold x_2^*)\in
\mathcal C}\exp\Bigl(-\frac{s^2}{2}\Bigr)s_1^ks_2^k\,d\bold x_1 d\bold x_2^*
=\iiint\limits_{ s=\Theta(n^{1/2})}\!\!\!\!\exp\Bigl(-\frac{s^2}{2}\Bigr)\,s_1^ks_2^k\\
&\times\prod_t\frac{s_t^{|\mathcal I_t|-1}}{(|\mathcal I_t|-1)!}\pr\!\left(\!\max_i L_i^{(t)}\le \min\left\{\frac{1}{s_t}, 
\frac{\log^2 n}{n}\right\};\,
\sum_i(L_i^{(t)})^2\le \frac{3}{|\mathcal I_t|}\!\right)\,d\bold s.
\end{align*}
Since $s_t=O\bigl(n^{1/2}\bigr)\ll n \log^{-2}n$ and $|\mathcal I_t|=\Theta(n)$, the $t$-th probability
factor is at least
\[
1-\exp\bigl(-\Theta(\log^2 n)\bigr)-\exp\bigl(-\Theta(n^{\gamma})\bigr), 
\]
for $\gamma\in (0,1/3)$, see Lemma \ref{sumsofLs} and \eqref{P(L^+>)<}. Since $|\mathcal I_1|=
|\mathcal I_2|=\nu$, $|\mathcal I_3|=n-\nu$, we see that, with $\eps_n:=e^{-\Theta(\log^2n)}$,
\begin{align*}
&\qquad\qquad\qquad\quad\frac{1}{k!}\idotsint\limits_{(\bold x_1,\bold x_2^*)\in
\mathcal C}\exp\Bigl(-\frac{s^2}{2}\Bigr)s_1^ks_2^k\,d\bold x_1 d\bold x_2^*\\
&\qquad\,\,\ge(1-\eps_n)\!\!\iiint\limits_{ s=\Theta(n^{1/2})}\!\!\!\!\exp\Bigl(-\frac{s^2}{2}\Bigr)\,\frac{s_1^{\nu+k-1}s_2^{\nu+k-1}s_3^{n-\nu-1}}{k! \bigl((\nu-1)!\bigr)^2(n-\nu-1)!}\,d\bold s\\
&\quad=\frac{(1-\eps_n)\bigl((\nu+k-1)!\bigr)^2}{k!\bigl((\nu-1)!\bigr)^2(n+\nu+2k-1)!}\int\limits_{s=\Theta(n^{1/2})}\!\!\!\!\!\!\!\!\!\exp\Bigl(-\frac{s^2}{2}\Bigr)\, s^{n+\nu+2k-1}\,ds,
\end{align*}
using \eqref{int,prod} for the last step. The integrand attains its sharply pronounced maximum at
$s_{\text{max}}=(n+\nu+2k)^{1/2}$, which is $\Theta(n^{1/2})$ for $k=O(n)$. Let us choose 
$J:=[s_{\text{max}}-\log n, s_{\text{max}}+\log n]$ as the range of $s$. Since
\[
\frac{d^2}{ds^2}\log\left[\exp\Bigl(-\frac{s^2}{2}\Bigr)\, s^{n+\nu+2k-1}\right]\le -1,
\]
it follows in a standard way (cf. the proof of Lemma \ref{IC3sim}) that
\begin{align*}
\int\limits_{s\in J}\!\!\exp\Bigl(-\frac{s^2}{2}\Bigr)\, s^{n+\nu+2k-1}\,ds&\ge (1-\eps_n)\int\limits_0^{\infty}
\exp\Bigl(-\frac{s^2}{2}\Bigr)\, s^{n+\nu+2k-1}\,ds\\
&\ge (1-\eps_n)(n+\nu+2k-1)!!.
\end{align*}
Therefore, using 
\[
\binom{b}{a}\le \frac{b^b}{a^a(b-a)^{b-a}},\quad m!=\Theta\left[\left(\frac{m}{e}\right)^m\right],
\quad  (m-1)!!=\Theta\left[\left(\frac{m}{e}\right)^{m/2}\right],
\]
we obtain 
\begin{equation}\label{iiint>}
\begin{aligned}
&\idotsint\limits_{\bold x_1,\, \bold x_2^*}\!\!
f(\bold x_1,\bold x_2)\,d\bold x_1 d\bold x_2^*
\gtrdot \frac{\bigl((\nu+k-1)!\bigr)^2}{k!\bigl((\nu-1)!\bigr)^2(n+\nu+2k-1)!!}\\
&\qquad\qquad\qquad\qquad\quad\quad\,\,\,\,\gtrdot n^{1/2} e^{H(\nu,k)},\\
&\quad H(\nu,k):=-k\log (ke)+2(\nu+k)\log(\nu+k)-2\nu\log\nu\\
&\qquad\qquad\qquad\quad\,\,\,-\frac{n+v+2k}{2}\log\frac{n+\nu+2k}{e}.
\end{aligned}
\end{equation}
Treating $k$ as a continuously varying parameter, we have
\begin{equation}\label{H'_k=}
H'_k(\nu,k)=2\log (\nu+k)-\log k-\log(n+\nu+2k)=\log\frac{(\nu+k)^2}
{k(n+\nu+2k)}.
\end{equation}
From \eqref{H'_k=} we see that $H(\nu,k)$ has a unique stationary point
\begin{align*}
k(\nu)&=\frac{2\nu^2}{n-\nu+\sqrt{\nu^2+(n-\nu)^2}}=n \phi(\xi),\quad \xi:=\frac{\nu}{n},\\
\phi(x)&:=\frac{2x^2}{1-x+\sqrt{x^2+(1-x)^2}}.
\end{align*}
So, using \eqref{H'_k=} again,
\begin{equation}\label{H(nu,k(nu))=}
\begin{aligned}
&H(\nu,k(\nu))=2\nu\log(\nu+k(\nu))-2\nu\log\nu-\frac{n+\nu}{2}\log\frac{n+\nu+2k(\nu)}{e}\\
&\quad=n\left[2\xi\log\left(1+\frac{\phi(\xi)}{\xi}\right)-\frac{1+\xi}{2}\left(\log\frac{n}{e}+\log(1+\xi+2\phi(\xi)\right)\right].
\end{aligned}
\end{equation}
Since only integers $k$ qualify for the bound \eqref{iiint>}, we introduce $k^*(\nu)=\lceil k(\nu)\rceil$.
As $H''_k(k,\nu)=O(n^{-1})$, we have $H(\nu,k^*(\nu))=H(\nu,k(\nu))+O(n^{-1})$.

For the first factor on the RHS of \eqref{single,simple} we have
\begin{equation}\label{first>}
\frac{(n!)^2}{(n-\nu)!}\gtrdot n^{1/2}\exp\left[-(1-\xi)\log(1-\xi) +(1+\xi)\log\frac{n}{e}\right].
\end{equation}
Combining the equations \eqref{iiint>}, \eqref{H(nu,k(nu))=} and \eqref{first>}, we obtain
\begin{equation}\label{ES_n^2>}
\ex\bigl[(S_n)_2\bigr]\gtrdot n^{3/2}\exp\bigl(n\mathcal H(\xi)\bigr),\quad \xi=\frac{\nu}{n},
\end{equation}
with $\mathcal H(\xi)$ defined in \eqref{H(xi)=}. As a function of the continuously varying $\xi\in (0,1)$,
$H(\xi)$ attains its maximum at $\xi_{\text{max}}\approx 0.739534$. Introduce $\nu^*=\lceil n\xi_{\text{max}}
\rceil$; then $\frac{\nu^*}{n}=\xi_{\text{max}}+O(n^{-1})$, implying that $H(\nu^*/n)=H(\xi_{\text{max}})+
O(n^{-1})$. Therefore
\[
\ex\bigl[(S_n)_2\bigr]\gtrdot n^{3/2}\exp\bigl(n\mathcal H(\xi_{\text{max}})\bigr).
\]
The proof of Theorem \ref{ESn^2>} is complete.

\section{Estimates for one-sided matchings}
\subsection{$\pr(M)$, $\ex[S_n]$, $\mathcal P(M)$} By Lemma \ref{P(Mest)='}, 
\[
\pr(M)=\pr(M\text{ is e-stable})=\idotsint\limits_{\bold x\in [0,1]^n}\prod_{(a,b\neq M(a))}(1-x_ax_b)\,d\bold x.
\]
Let 
\[
\mathcal C^*=\Biggl\{\bold x \in C_3:\,\Bigl|\frac{2n\sum_{i\in [n/2]}x_ix_{i+n/2}}{s^2}-1\Bigr|\le n^{-\sigma}\Biggr\},
\]
where $C_3$ is defined in Corollary \ref{C3,expl}.  Very similarly to Lemma \ref{prodsim}, uniformly for $x\in 
\mathcal C^*$, we have
\[
\prod_{(a, b\neq M(a))} (1-x_i x_j)=\exp\!\left(\!-\frac{s^2}{2}\left(\!1-\frac{3}{n}\!\right)-\frac{s^4}{n^2}+O(n^{-\sigma})\!\right).
\]
And, just like Corollary \ref{C3,expl} itself, invoking the bound \eqref{sumL_jL_{j+}} we obtain
\[
\idotsint\limits_{\bold x\in [0,1]^n\setminus \mathcal C^*}\prod_{(a,b\neq M(a))}(1-x_ax_b)\,d\bold x\le 
\frac{e^{-\Theta(\log^2n)}}{(n-1)!!}.
\]
Arguing as in the proof of Lemma \ref{IC3sim}, and using $\ex[S_n]=(n-1)!! \pr(M)$, we establish
\begin{Theorem}\label{IC3sim'}
\begin{align*}
\pr(M)&=\bigl(1+O(n^{-\sigma})\bigr)\frac{e^{1/2}}{(n-1)!!},\\
\ex[S_n]&=e^{1/2}+O(n^{-\sigma}),\quad\forall \sigma<1/3.
\end{align*}
\end{Theorem}
{\bf Note.\/} Since $\ex[S_n]$ also equals the expected number of the usual, one-sided, stable
matchings, we actually gave here a corrected proof of our result from \cite{Pit2}. See the note following
Lemma \ref{sumLj^s}.

Turn to $\mathcal P(M)$, the probability that $M$ is both stable and e-stable. By Lemma \ref{P(Mest)='}
\eqref{simple1,2}, Lemma \ref{intervals1} and \eqref{int,prod},
\begin{align*}
\mathcal P(M)&=\idotsint\limits_{\bold x\in [0,1]^n}\prod_{(a,b\neq M(a))}(1-x_ax_b)^2\,d\bold x\\
&\lessdot \idotsint\limits_{\bold x\ge\bold 0} e^{-s^2}\,d\bold x\le \int_0^{\infty}e^{-s^2}\frac{s^{n-1}}{(n-1)!}\,ds
\\
&=\frac{2^{-n/2}}{(n-1)!!}
\end{align*}
Since the total number of matchings on $[n]$ is $(n-1)!!$, we have proved
\begin{Theorem}\label{nodouble'}
\[
\pr(\exists\, M\text{ both stable and e-stable})= O\bigl(2^{-n/2}\bigr).
\]
\end{Theorem}
\subsection{Likely range of the partners's ranks} Let $R(M)$ be the sum of all partners's ranks under
$M$, and $\!\!\pr_k(M)=\!\!\pr(M\text{ is e-stable}, R(M)\!=k)$.  By Lemma \ref{Pk(M)'},
\[
\pr_k(M)=\!\!\idotsint\limits_{\bold x\in [0,1]^{n}}[\xi^{k-n+1}]\!\!\prod_{(a,b\neq M(a))}
\!\!\!\!\!\!\bigl(\bar x_a\bar x_b+\xi x_a\bar x_b+\xi \bar x_a x_b\bigr)\,d\bold x;
\]
this integral is also the probability that $M$ is stable, and $R(M)=k$. 
\begin{Theorem}\label{R(M)appr'} For a fixed $\eps\in (0,1)$,
\[
\pr\left(\max_M\left|\frac{R(M)}{n^{3/2}}-1\right|\ge \eps\right)\le e^{-\Theta(\log^2 n)}.
\]
\end{Theorem}
In our recent \cite{Pit3} we proved the similar result for the total rank of ``predecessors'' in the stable cyclic partitions, that include the stable matchings as a special case. Since the proof was based on the union
bound involving the distribution of that rank for a generic cyclic partition, Theorem \ref{R(M)appr'} is
a direct corollary of that result. The note following Theorem \ref{IC3sim'} could be replicated here.
\subsection{$\ex[S_n^2]$ and such}  
\begin{Theorem}\label{ESn^2>'} $\ex\bigl[S_n^2\bigr]\gtrdot\exp\Bigl[\frac{n}{2}H(\xi_{\text{max}})]> 
1.13^n$, for $n$ large enough.
Consequently, for each such $n$ there exists an instance of the $n$ preference lists with the number of e-stable matchings exceeding $1.06^n$.
\end{Theorem}
\begin{proof}
The one-sided counterpart of \eqref{single} is 
\begin{equation}\label{single'}
\begin{aligned}
\ex\bigl[(S_n)_2\bigr] &\ge \mathcal B(n,\nu)
\,\,\,\idotsint\limits_{\bold x_1\in [0,1]^n,\,\,\bold x_2^*\in [0,1]^{\nu}}\!\!\!\!\!\!\!\!\!
f(\bold x_1,\bold x_2)\,d\bold x_1 d\bold x_2^*.
\end{aligned}
\end{equation}
Here $\mathcal B(n,\nu)$ is the total number of pairs $(M_1,M_2)$ of general matchings $M_1$ and $M_2$, 
on $[n]$ with $\nu(M_1,M_2)$, the total length of circuits formed by the pairs from $M_1\Delta M_2$, equal
to $\nu$. 
 More explicitly, we have $\mathcal B(n,\nu)=\binom{n}{\nu}(n-\nu-1)!! \mathcal B(\nu)$.
Here $\mathcal B(\nu)$ is the total number of the disjoint unions of {\it circuits\/} on the vertex set $[\nu]$,
with every second edge on each circuit marked as belonging to the matching $M_1$, and the intervening edges being assigned to the matching $M_2$. Thus $\mathcal B(\nu)$ is also the total number of  permutations on $[\nu]$ with {\it cycles\/} of length $\ge 4$. Since the total number of those permutations
with $r_j$ cycles of length $j\ge 4$ is $\nu! \prod_j \frac{1}{(j!)^{r_j}r_j!}$, it follows easily
that
\begin{align*}
\sum_{\nu\ge 4}x^{\nu}\,\frac{\mathcal B(\nu)}{\nu!}&=\exp\left(\sum_{\text{even }j\ge 4}\frac{x^j}{j}\right)
=\frac{e^{-\frac{x^2}{2}}}{(1-x^2)^{1/2}}.
\end{align*}
Using the saddle-point method (Flajolet and Sedgewick \cite{FlaSed}), we obtain
\[
\mathcal B(\nu)=\bigl(1+O(\nu^{-1})\bigr)\nu!\sqrt{\frac{2}{\pi e\nu}}\gtrdot \nu!\,\nu^{-1/2}.
\]
Consequently, for $\nu=\Theta(n)$, 
\begin{equation}\label{first>'}
\begin{aligned}
\mathcal B(n,\nu)&
\gtrdot \exp\left[\frac{n}{2}\left(-(1-\xi)\log(1-\xi)+(1+\xi)\log\frac{n}{e}\right)\right],\,\,\xi:=\frac{\nu}{n};
\end{aligned}
\end{equation}
cf. \eqref{first>}.
Combining the equations \eqref{single'}, \eqref{iiint>}, \eqref{H(nu,k(nu))=} and \eqref{first>'}, we obtain
\begin{equation}\label{ES_n^2>'} 
\ex\bigl[(S_n)_2\bigr]\gtrdot n^{3/2}\exp\Bigl(\frac{n}{2}\mathcal H(\xi)\Bigr),\quad \xi=\frac{\nu}{n},
\end{equation}
with $\mathcal H(\xi)$ defined in \eqref{H(xi)=}. The rest follows the conclusion of the proof of Theorem
\ref{ESn^2>}.
\end{proof}
{\bf Acknowledgment.\/} I am grateful to David Manlove for bringing the existing work
on the doubly stable matchings to my attention, and asking how likely these matchings are.

\begin{thebibliography}{99}
\bibitem{Alc}
J. Alcalde, \textit{Exchange-proofness or divorce proofness? Stability in one-sided matching markets}, Economic Design \textbf{1} (1995) 275--287.
\bibitem{AndAskRoy}
G. E. Andrews, R. Askey and R. Roy, \textit{Special Functions}, Cambridge University Press (1999).
\bibitem{AshKanLes}
I. Ashlagi, Y. Kanoria and J. D. Leshno, \textit{Unbalanced Random Matching Markets: the Stark Effect of Competition},  J. Polit Economy, \textbf{125}  (2017) 69--98.
\bibitem{CecMan}
K. Cechl\'arov\'a and D. F. Manlove, \textit{The exchange-stable marriage problem}, Disc. Appl. Math.
\textbf{152} (2005) 109--122.
\bibitem{FlaSed}
P. Flajolet and R. Sedgewick, \textit{Analytic Combinatorics}, Cambridge University Press (2009).
\bibitem{GalSha}
D. Gale and L. S. Shapley, \textit{College admissions and the stability of marriage}, Amer Math
Monthly \textbf{69} (1962) 9--15.
\bibitem{GusIrv}
D. Gusfield and R. W. Irving, \text{The Stable Marriage Problem, Structure and Algorithms}, The MIT
Press (1989).
\bibitem{IrvPit}
R. W. Irving and B. Pittel, \textit{An upper bound for the solvability probability of a random stable roommates instance}, \textbf{5} (1994) 465--486.
\bibitem{Irv1}
R. W. Irving, \textit{Stable matching problems with exchange restrictions},  Combin Optimization
\textbf{16} (2008) 344--360.
\bibitem{Knu}
D. E. Knuth, \textit{Stable marriage and its relation to other combinatorial problems: an introduction
to the mathematical analysis of algorithms}, CRM Proceedings and Lecture notes (1996).
\bibitem{KnuMotPit}
D. E. Knuth, R. Motwani and B. Pittel, \textit{Stable husbands}, Random Struct Algorithms \textbf{1}
(1991) 1--14.
\bibitem{LenPit}
C. Lennon and B. Pittel, \textit{On the likely number of solutions for the stable matching problem},
Combin Probab Comput \textbf{18} (2009) 371--421.
\bibitem{McDCheSuz}
E. McDermid, C. Cheng and I. Suzuki, \textit{Hardness results on the man-exchange stable marriage 
problem with short preference lists}, Inform. Proc. Letters \textbf{101} (2007) 13--19.
\bibitem{McVWil}
D. G. McVitie and L. B. Wilson, \textit{The stable marriage problem}, Comm ACM \textbf{14} (1971)
486--490.
\bibitem{Man}
D. F. Manlove, \textit{Algorithmics of Matching under Preferences}, World Scientific (2013).
\bibitem{Mer}
S. Mertens, \textit{Random stable matchings}, J. Statist. Mechanics: Theory and Experiment \textbf{10}
(2005).
\bibitem{Pit5}
B. Pittel, \textit{A simple probability model of collective behavior}, Problemy Pereda\^ci Informacii \textbf{3},
(1967) 37--52 (Russian); translated as  \textit{Problems of Information Transmission} \textbf{3} (1969)
 30--44.
\bibitem{Pit0}
B. Pittel, \textit{The average number of stable matchings}, SIAM J Disc Math \textbf{2} (1989)
530--549.
\bibitem{Pit1}
B. Pittel, \textit{On a random instance of a ``stable roommates'' problem: Likely behavior of the proposal
algorithm}, Comb Probab Comput \textbf{2} (1993) 53--92.
\bibitem{Pit2}
B. Pittel, \textit{The ``stable roommates'' problem with random preferences},  Ann Probab  \textbf{21}  (1993) 1441--1477.
\bibitem{Pit4}
B. Pittel, \textit{On likely solutions of the stable matching problem with unequal numbers of
men and women}, available at arXiv:1701.08900 .
\bibitem{Pit3}
B. Pittel, \textit{On random stable partitions}, available at arXiv:1705.08340.
\bibitem{Ren}
A. R\'enyi, \textit{Probability Theory}, North-Holland (1970).
\bibitem{Tan}
J. J. M. Tan, \textit{A necessary and sufficient condition for the existence of a complete stable matching}, J. Algorithms \textbf{12} (1991) 1--25.
\bibitem{Tan1}
J. J. M. Tan, \textit{Stable matchings and stable partitions}, International J. Computer Math.
\textbf{39} (1991) 11--20.
\bibitem{Tse}
M. L. Tsetlin, \textit{Automaton theory and modeling of biological systems}, Translated from Russian by Scitran (Scientific Translation Service), Mathematics in Science and Engineering,  \textbf{102} Academic Press (1973). 
\bibitem{Wil} 
L. B. Wilson, \textit{An analysis of the stable marriage assignment problem}, BIT \textbf{12} (1972) 569--575.
\end{thebibliography}
\end{document}